\documentclass[11pt]{amsart}

\usepackage{amsmath,amssymb}

\usepackage{hyperref}

\numberwithin{equation}{section}
\theoremstyle{plain}
\newtheorem{theorem}[equation]{Theorem}

\newtheorem{proposition}[equation]{Proposition}
\newtheorem{lemma}[equation]{Lemma}
\newtheorem{corollary}[equation]{Corollary}

\theoremstyle{remark}

\newtheorem{remark}[equation]{Remark}
\theoremstyle{definition}
\newtheorem{definition}[equation]{Definition}

\newtheorem{example}[equation]{Example}

\newcommand{\lra}{\longrightarrow}
\newcommand{\ra}{\to}

\newcommand{\C}{{\mathcal C}}

\newcommand{\D}{{\mathcal D}}

\newcommand{\G}{{\mathcal G}}

\newcommand{\I}{{\mathcal I}}
\newcommand{\K}{{\mathcal K}}
\renewcommand{\L}{{\mathcal L}}

\newcommand{\N}{\mathbb N}
\renewcommand{\P}{{\mathcal P}}

\newcommand{\R}{{\mathbb R}}

\newcommand{\U}{{\mathcal U}}

\newcommand{\Z}{\mathbb Z}

\newcommand{\Lip}{\operatorname{Lip}}

\newcommand{\Star}{\operatorname{St}}

\newcommand{\halfstar}{\operatorname{St}_{\frac12}}

\def\De{\Delta}

\def\ga{\gamma}
\def\Ga{\Gamma}

\def\lra{\longrightarrow}

\def\ra{\to}

\def\si{\sigma}
\def\Si{\Sigma}

\def\be{\beta}

\def\th{\theta}

\def\defeq{:=}

\newcommand{\ol}{\overline}
\newcommand{\no}{\noindent}

\def\Xint#1{\mathchoice
    {\XXint\displaystyle\textstyle{#1}}%
    {\XXint\textstyle\scriptstyle{#1}}%
    {\XXint\scriptstyle\scriptscriptstyle{#1}}%
    {\XXint\scriptscriptstyle\scriptscriptstyle{#1}}%
    \!\int}
\def\XXint#1#2#3{{\setbox0=\hbox{$#1{#2#3}{\int}$}
      \vcenter{\hbox{$#2#3$}}\kern-.5\wd0}}

\def\av{\Xint-}

\begin{document}

\begin{abstract}
We give conditions on Gromov-Hausdorff convergent inverse systems 
of metric measure graphs 
which imply that the measured Gromov-Hausdorff limit
(equivalently, the inverse limit)
 is a PI space i.e., it satisfies   a doubling condition and
 a Poincar\'e  inequality in the sense of Heinonen-Koskela \cite{HeKo}. The Poincar\'e inequality
is actually of type $(1,1)$.  We also give a systematic construction of
examples for which our conditions are satisfied.
 Included are known examples of 
PI spaces, such as Laakso spaces, and a large class of new examples. 

As follows easily from \cite{ckdppi}, generically our examples have
the property that they do not bilipschitz embed in any Banach space with
Radon-Nikodym property. For Laakso spaces, this was noted in \cite{ckdppi}.
However according to \cite{ckellone} these spaces 
admit a bilipschitz embedding in $L_1$. For Laakso spaces, this 
was announced in \cite{ckbv}.
\end{abstract}

\title[Inverse limits 
satisfying a Poincar\'e inequality]{Inverse limit spaces satisfying a Poincar\'e inequality}
 
\date{\today}
\author{Jeff Cheeger and Bruce Kleiner}
\thanks{J.C. was supported by NSF Grant DMS-1005552 and by a Simons Fellowship,
and B.K. was supported by NSF Grant DMS-1105656.}

\maketitle

{\small\tableofcontents}

\section{Introduction}
\label{intro}

This paper is part of a series concerning bilipschitz embeddability and
PI spaces, i.e. metric measure spaces which
satisfy a doubling condition and a Poincar\'e inequality; 
\cite{crannouncement},
\cite{GFDA},  \cite{ckdppi}, \cite{CKN}, \cite{ckbv},
\cite{ckmetmonellone}, \cite{MR2892612}, \cite{ckellone}. 
In this paper we give a systematic construction of 
PI spaces as 
 inverse limits, or equivalently Gromov-Hausdorff limits, 
 of certain inverse systems of metric measure graphs which we term ``admissible''
 (see Section \ref{prelim} for the definition).  Included are known examples of 
PI spaces, such as Laakso spaces (\cite{laakso}) and a large class of new examples. 

Our main result is:

\begin{theorem}
\label{t:pi}
The  measured Gromov-Hausdorff
limit  of an admissible inverse system 
is a PI space satisfying 
a $(1,1)$-Poincar\'e inequality. Moreover, the doubling constant $\beta$ and the constants $\tau,\Lambda$
in the Poincar\'e inequality depend only on the constants
$2\leq m\in \N$, $\De,\th,C\in (0,\infty)$ in conditions (1)--(6)
for admissible inverse systems.
\end{theorem}

The  limit spaces have analytic dimension 1,
topological dimension 1 and except in certain ``degenerate'' cases, Hausdorff dimension $>1$.
It  follows from \cite{ckellone} that the spaces we construct admit bilipschitz embeddings
in $L_1$.  For Laakso spaces, this 
was announced in \cite{ckbv}.
However, except in the degenerate cases, they do not bilipschitz
embed in any Banach space with the Radon-Nikodym Property.
For Laakso spaces, this was noted in \cite{ckdppi}.

One of the novelties in this paper is a new approach to proving the Poincar\'e inequality
that exploits the fact that the metric measure space is the limit of 
an inverse system 
$$
X_0\stackrel{\pi_0}{\longleftarrow}\cdots \stackrel{\pi_{i-1}}{\longleftarrow}X_{i}
\stackrel{\pi_{i}}{\longleftarrow}
\cdots\,.
$$
The argument, which is by induction, 
involves averaging a function on $X_{i+1}$
over the fibers of the projection 
map $\pi_i:X_{i+1}\ra X_i$, to produce a function on $X_i$.  The averaging operator is
defined by specifying, for each $x\in X_i$,  a probability measure $\D_i(x)$ 
supported on the fiber $\pi_i^{-1}(x)\subset X_{i+1}$; for a generic point $x\in X_i$, the 
choice of $\D_i(x)$ is canonical.  The key point is that under
a certain condition (see Axiom (6) from Definition \ref{admissible})
this canonical
assignment extends to one that is continuous with respect to the 
weak topology on Radon measures, and that is compatible with the operation of
taking upper gradients.  This new proof of the Poincar\'e inequality is robust and
applies verbatim to certain
higher dimensional inverse systems.

\subsection*{Organization of the paper}
 
In  Section \ref{prelim}, after we recall some standard material, we 
state the six axioms which define admissible inverse systems, discuss the  
role of the  axioms, and draw some simple consequences.  
Among a number of other things,
we show in Corollary \ref{cor_top_dim_1} that the topological dimension of the 
inverse limit
is $1$.

In Section \ref{s:blg}, for each $X_i$,
we verify, with uniform constants, the  Poincar\'e inequality locally
at the scale associated with $X_i$, as well as  the (global)
doubling condition. 

In Section \ref{spi},
the last three axioms are reformulated in terms of what we call 
``continuous fuzzy sections''
of the maps $\pi_i: X_{i+1}\to X_i$ of our inverse system.
This reformulation plays a role in several places in the paper.

In Section \ref{s:piproof}, using the continuous fuzzy sections,
 we prove that the $X_i$'s satisfy a uniform Poincar\'e inequality; this  
implies 
that the Gromov-Hausdorff limit $X_\infty$ has a Poincar\'e inequality (\cite{cheeger,keith})
thereby proving Theorem \ref{t:pi}.

In Section Section \ref{mlp} we construct a natural probability measure on the  family of paths
in $X_k$ which are lifts of some fixed path in $X_j$ ($j<k$).

In Section \ref{s:p2}, we give a  second, essentially different, 
proof of the Poincar\'e inequality for $X_\infty$ using the probability measure on path families.

In Section \ref{eym} we show how to construct large  families of examples of admissible 
inverse systems.  The construction produces a sequence of partial inverse systems
$$
X_0\stackrel{\pi_0}{\longleftarrow}\cdots \stackrel{\pi_{i-1}}{\longleftarrow}X_{i}
$$
by induction on $i$; in the  inductive step, roughly speaking, one makes independent choices 
locally in $X_i$ to produce $X_{i+1}$.
Both fuzzy sections and the
measure on path families  play a role in the discussion.

In Section \ref{s:ad} we show that for an admissible inverse 
system, the 
cotangent bundle of the limit has dimension $1$.

In Section  \ref{s:blnernp}, we show that except in degenerate cases, limits of 
admissible systems do
not bilipschitz embed in any Banach space with the Radon-Nikodym Property.

  In Section \ref{higher} we  briefly indicate how our previous discussion
can be extended to 
certain higher dimensional inverse systems. In this case,
 depending which building blocks one uses, for example the Heisenberg group 
  with its Carnot-Caratheodory metric,
   the resulting inverse limit spaces need not bilipschitz embed in $L_1$.

\section{Preliminaries}
\label{prelim}

In this section we begin by collecting some standard definitions.  Then we give the axioms
for an admissible inverse
system, briefly indicate the role of  each of the axioms and observe some elementary consequences.

\subsection{The doubling condition and the Poincar\'e inequality}

We now recall some relevant definitions.
Let $(X,d,\mu)$ denote a metric measure space, with
$\mu$ a Borel  measure on $X$, which is finite and nonzero
on metric balls $B_r(x)$ if $0<r<\infty$.

For $U$ measurable, we set
\begin{equation}
\label{av2}
f_U=\frac{1}{\mu(U)}\int_Uf\, d\mu\, .
\end{equation}

The measure $\mu$ is said to satisfy
a {\bf doubling condition} if there exists $\beta=\beta(R)$ such that for all $x\in X$ 
\begin{equation}
\label{doubling}
\mu(B_{2r}(x))\leq\beta\cdot\mu(B_r(x))\, , \qquad(r\leq R)\, .
\end{equation}

If $(X,d)$ is a metric space, $f:X\to\R$ and a nonnegative Borel function  $g:X\to \R_+$, we say that $g$ is an
{\bf upper gradient} for $f$ if for all rectifiable curves $c: [0,L]\to X$ parameterized by
arclength,
\begin{equation}
\label{igdf}
|f(c(L))-f(c(0))|\leq \int_0^Lg(c(s))\, ds\, .
\end{equation}

We say that $(X,d,\mu)$ {\bf satisfies a $(1,p)$-Poincar\'e inequality}
 if for some
$\Lambda$ and $\tau=\tau(R)$, we have for every bounded
continuous function $f$ and every upper
gradient $g$,
\begin{equation}
\label{11pi1}
\int_{B_r(x)}|f-f_{B_r(x)}|\, d\mu\leq \tau r\cdot
\left(\int_{B_{\Lambda r}(x)}(g)^p\, d\mu\right)^{\frac{1}{p}}
\qquad
(r\leq R)\, .
\end{equation}
This definition and the definition  of upper gradient are
due to Heinonen-Koskela \cite{HeKo}.

It was shown in    \cite[Theorem 1.3.4]{keith} that $(X,d,\mu)$
satisfies a $(1,p)$-Poincar\'e inequality if and only if for every Lipschitz function $f$,
\begin{equation}
\label{11pi}
\int_{B_r(x)}|f-f_{B_r(x)}|\, d\mu\leq \tau r\cdot
\left(\int_{B_{\Lambda r}(x)}(\Lip\, f(x))^p\, d\mu\right)^{\frac{1}{p}}
\qquad
(r\leq R)\,,
\end{equation}
where $\Lip f$ denotes the pointwise Lipschitz constant of $f$:
$$
\Lip\, f(x)\defeq\limsup_{d(x',x)\to 0}\, 
\frac{|f(x')-f(x)|}{d(x',x)}\, \qquad (x'\ne x)\, .
$$

\begin{definition}
\label{dpi}
If (\ref{doubling}) and (\ref{11pi1}) hold, we say that $(X,d,\mu)$ is a PI space.
\end{definition}

\begin{remark}
The examples constructed in this paper will satisfy $p=1$, which is the
strongest version of the Poincar\'e inequality.
\end{remark}

\subsection{Axioms for admissible inverse systems}
\label{axioms}

We will  consider inverse systems of connected
 metric measure graphs, 
\begin{equation}
\label{is}
X_0\stackrel{\pi_0}{\longleftarrow}\cdots \stackrel{\pi_{i-1}}{\longleftarrow}X_{i}
\stackrel{\pi_{i}}{\longleftarrow}
\cdots\, .
\end{equation}

Let $\Star(x,G)$ denote the star
of a vertex $x$ in a graph $G$, i.e. the union of the edges containing $x$.

We assume that each $X_i$ is connected 
and is equipped with a  path metric $d_i$ and a measure
$\mu_i$, such that the 
following conditions hold, for some constants
$2\leq m\in \Z$, $\De$, $\th$, $C\in (0,\infty)$ and
every $i\in \Z$\,:

\begin{enumerate}
\setlength{\itemsep}{.5ex}
\item   (Bounded local metric geometry) $(X_i,d_i)$ is a nonempty connected 
 graph with all vertices of valence $\leq \De$, and such that every 
edge of $ X_i$ is  isometric to an interval of length $m^{-i}$
with respect to the path metric $d_i$. 
\item (Simplicial projections are open)  If $X_i'$ denotes the  graph obtained by
subdividing each edge of $X_i$ into
$m$ edges of length $m^{-(i+1)}$, then $\pi_i$ induces a map
$\pi_i:(X_{i+1},d_{i+1})\ra (X_i',d_i)$ which is open, simplicial, and an isometry on every edge.
\item (Controlled fiber diameter)
For every  $x_i\in X_i'$, the inverse image $\pi_i^{-1}(x_i)\subset X_{i+1}$ 
has $d_{i+1}$-diameter  at most $\th \cdot m^{-(i+1)}$.  
\item (Bounded local metric measure geometry.)
The measure $\mu_i$ restricts to a constant multiple of arclength on each edge 
$e_i\subset X_i$, 
and  $\frac{\mu_i(e_{i,1})}{\mu_i(e_{i,2})}\in [C^{-1},C]$ for any two adjacent edges 
$e_{i,1},e_{i,2}\subset X_i$.
\item (Compatibility with projections)  
$$
(\pi_i)_*(\mu_{i+1})=\mu_i\, ,
$$
where $(\pi_i)_*(\mu_{i+1})$ denotes the pushforward of $\mu_{i+1}$ under $\pi_i$.
\item (Continuity)  For all vertices $v_i'\in X_i'$, and $v_{i+1}\in \pi_i^{-1}(v_i')$,
the quantity
\begin{equation}
\label{con1}
\frac{\mu_{i+1}(\pi_i^{-1}(e_i')\cap \Star(v_{i+1},X_{i+1}))}{\mu_i(e_i')}
\end{equation}
is the same for all edges $e_i'\in \Star(v_i',X_i')$.   
\vskip2mm

\end{enumerate}

\begin{definition}
\label{admissible}
An inverse system of  metric measure graphs as in (\ref{is})  is called {\bf admissible} 
if it satisfies (1)--(6).
\end{definition}
\vskip2mm


\subsection{Discussion  of the axioms and some elementary consequences}
 Let us give a brief indication of the relevant consequences of each of our axioms.
Note that the first three axioms deal only with the metric and not the measure.
Indeed, taken together, Axioms (1) and (2)  have the  following 
purely combinatorial content which is worth noting at the outset, since it helps
to picture the restricted class of inverse systems that we are considering.
\begin{proposition}
\label{1-1}
Let $\{v_i\}$ denote a compatible sequence of vertices, i.e. $v_i$ is a vertex of 
$X_i$ and  $\pi_i(v_{i+1})=v_i$, for all $i\geq 0$.
Then for all but at most $\Delta$ values of $i$, the restriction of the
 locally surjective map $\pi_{i}$ to the open star of $v_{i+1}$ 
is actually $1$-$1$.
\end{proposition}
\begin{proof}
 From the local surjectivity of $\pi_{i}$ it follows that  the number of edges
emanating from $v_i$ is a nondecreasing function of $i$. Therefore, from the
uniform bound $\Delta$ on the degree of a vertex, of $X_i$, for all $i$, 
the proposition follows.
\end{proof}

Axiom (1) includes the statement that $\pi_i:X_{i+1}\to X'_i$
is a finite-to-one simplicial map.
This  implies that the vertices
of $X_{i+1}$ are precisely the inverse images of vertices of $X'_{i}$.
The second part of  Axiom (1) states that the restriction of $\pi_i$
to every edge is an isometry. In particular,
$\pi_{i}:(X_{i+1},d_{i+1})\to (X_{i},d_{i})$
is $1$-Lipschitz, i.e. distance nonincreasing. Axiom 
(1) also implies that for all $K>0$, if the ball in $X_i$ of radius $\leq K\cdot m^{-i}$
is rescaled to unit size,  then the  metric geometry
has a uniform bound depending on $K$ but independent of $i$.
\vskip2mm

 Axiom (2), stating that $\pi_{i}$ is open, implies
that if $c$ is a rectifiable path parameterized
by arc length and $\pi_{i}(x_{i+1})=c(0)$, then there exists a lift
 $\widetilde c$ parameterized by arc length, with $\widetilde c(0)=x_{i+1}$.
In general, $\widetilde c$ is not unique. By Axiom  (1), the paths 
$c$ and $\widetilde c$ have equal lengths and in addition,
for all $i\geq 0$, $x_{i+1}\in X_{i+1}$ and $r>0$, we have
\begin{equation}
\label{surj}
\begin{aligned}
\pi_i(B_r(x_{i+1}))&=B_r(\pi_i(x_{i+1}))\, , \\
B_r(x_{i+1})&\subset\pi_i^{-1}(B_r(\pi_i(x_{i+1})))\, .
\end{aligned}
\end{equation}
 Axiom (2) is actually a consequence of
Axioms (4), (5) below.

Axiom (3), together 
with (\ref{surj}), gives
\begin{equation}
\label{containment}
B_r(\pi_i(x_{i+1}))\subset \pi_i^{-1}(B_{r+\theta m^{-(i+1)}}(x_{i+1}))
\subset B_{r+\theta m^{-(i+1)}}(\pi_i(x_{i+1}))\, .
\end{equation}

This statement, which can be iterated, says that inverse images of balls are themselves
comparable to balls. It  is used in the  inductive arguments which
control the constants in the doubling and Poincar\'e inequalities.

Axioms (1)--(3) imply that for all $x_{i+1,1}, x_{i+1,2}\in X_{i+1}$, we have 
\begin{equation}
\label{eqn_distance_growth}
\begin{aligned}
d_i(\pi_i(x_{i+1,1}),\pi_i(x_{i+1,2}))&\leq d_{i+1}(x_{i+1,1}, x_{i+1,2})\\
                                                       &\leq d_i(\pi_i(x_{i+1,1}),\pi_i(x_{i+1,2}))
+2\theta\cdot m^{-(i+1)}\, ;
\end{aligned}
\end{equation}
compare (\ref{surj}), (\ref{containment}).

Note also that Axioms (1) and (3) together imply that 
for all $i$ and all $x_i\in X_i$ the cardinality ${\rm card}(\pi_i^{-1}(x_i))$ satisfies
\begin{equation}
\label{cb}
{\rm card}(\pi_i^{-1}(x_i))\leq \Delta^{\theta+1}\, ,
\end{equation}
since any two points of $\pi_i^{-1}(x_i)$ are connected by an edge path of length $\leq\theta \cdot m^{-(i+1)}$
and there are at most $ \Delta^{\theta+1}$ such paths which start at a give point of $\pi_i^{-1}(x_i)$.

Axiom (4) implies that on scale $m^{-i}$ the metric measure geometry
of $X_i$ is bounded.  As a consequence, for balls $B_{cm^{-i}}(x_i)\subset X_i$
there is a doubling condition 
and Poincar\'e inequality with constants which depend only on $c$ and are independent
of $i$;
see for example Lemma \ref{lpi}.

Axiom (5) is used is showing that the sequence $(X_i,d_i,\mu_i)$
converges in the measured Gromov-Hausdorff sense.
It also plays a role in the inductive arguments verifying
the doubling condition and the Poincar\'e inequality. 

Axiom  (6) is the least obvious of our axioms. However, it enters crucially in both of the 
  proofs that we give of the  bound on the constant in 
the  Poincar\'e inequality for $(X_\infty,d_\infty,\mu_\infty)$; see Sections \ref{s:piproof}--\ref{s:p2}.
Here is a very brief indication of the role of Axiom (6).   Given Axioms (1)--(5), 
 the  disintegration $x\mapsto \D_i(x)$ of the measure $\mu_{i+1}$ with respect to the
mapping $\pi_i:X_{i+1}\ra X_i$,  can be used to push a function 
$f_{i+1}:X_{i+1}\to\R$ down to a function 
$f_i:X_i\to\R$.  If $f_{i+1}$ is Lipschitz, then 
Axiom (4) implies that away from the vertices of $X_i'$, the pointwise
Lipschitz constant of $f_i$ is controlled by that of $f_{i+1}$.
It follows from Axiom (6) that $f_i$ is continuous at vertices, and hence
the Lipschitz control holds at the vertices of $X_i'$ as well. 
This construction is a key part of the induction step in 
our first proof of the
Poincar\'e inequality.
(Absent Axiom (6), even if $f_{i+1}$ is Lipschitz,
the function $f_i$ need not be continuous at the vertices
of $X_i$.)

Dually, given Axioms (1)--(5), there is a natural probability measure $\Omega$ on the collection $\Gamma$
of lifts to $X_{i+1}$ of an edge path $\gamma_i'\subset X_i'$. If 
Axiom (6) holds, this measure  has the additional property of being
 independent of the orientation of $\gamma_i'$.  This  
turns out to be required for the proof of the Poincar\'e inequality based on path families.

\subsection{The inverse limit}

We recall that the {\bf inverse limit} of the inverse system $\{X_i\}$ is the 
collection $X_\infty$
of compatible sequences, i.e. 
$$
X_\infty=\{(v_i)\in \prod_i X_i \;\mid\;   \pi_i(v_{i+1})=v_i \;
\text{for all} \; i\geq 0\}\,.
$$
  For all $i\geq 0$, one has a projection 
map $\pi^\infty_i:X_\infty\ra X_i$ that sends $(v_j)\in X_\infty$ to $v_i$.  

For any $(v_i),(w_i)\in X_\infty$,
the sequence $\{d_j(v_j,w_j)\}$ is nondecreasing since 
the projection maps $\{\pi_j\}$ are $1$-Lipschitz, and bounded above by 
(\ref{eqn_distance_growth}); therefore we have a well-defined metric
on the inverse limit  given by 
$$
d_\infty((v_i),(w_i))=\lim_{j\ra \infty}d_j(v_j,w_j)\,.
$$
The projection map $\pi^\infty_i:(X_\infty,d_\infty)\ra (X_i,d_i)$ is $1$-Lipschitz.

We now record a  consequence of the above discussion:
\begin{corollary}
\label{cor_top_dim_1}
The inverse limit $X_\infty$ has topological dimension $1$. 
\end{corollary}
\begin{proof}
By the path lifting argument in the discussion of Axiom (2), 
one may take an edge $\ga_0\subset X_0$,
and lift it  isometrically
 to a compatible family $\{\ga_j\subset X_j\}_{j\geq 0}$
which produces a 
geodesic segment in $X_\infty$.  Therefore $X_\infty$ has topological dimension at
least $1$.

If $\U_i$ is the cover of $X_i$
by open stars of vertices, and $\hat\U_i$ is the inverse image of $\U_i$ under
the projection map $X_\infty\ra X_i$, then $\hat\U_i$ has $1$-dimensional nerve,
and the diameter of each open set $U\in\hat\U_i$ is $\lesssim m^{-i}$, see (\ref{containment}).
For any compact subset $K\subset X_\infty$, and any open cover $\U$ of $K$,
some $\hat\U_i$ will provide a refinement of $\U$; this shows that $K$ has topological
dimension $\leq 1$.  As $X_\infty$ is locally compact, it follows that $X_\infty$
has topological dimension $\leq 1$. 
\end{proof}

We now discuss the measure on $X_\infty$.
For every $i$, one obtains a subalgebra $\Si_i$ of the Borel $\si$-algebra on $X_\infty$
by taking the inverse image of the Borel $\si$-algebra on $X_i$.
One readily checks using (\ref{containment}) that the
$\si$-algebra generated by the countable union $\cup_i\, \Si_i$ is the full Borel 
$\si$-algebra on $X_\infty$.  The $\si$-algebra
$\Si_i$ has a measure $\hat\mu_i$ 
induced from $\mu_i$ by the projection $\pi^\infty_i$.  Axiom (5) implies that the
measures $\hat\mu_i$ on the
increasing family $\{\Si_i\}$ are compatible under restriction, and
by applying the Caratheodory extension theorem, one gets that the $\hat\mu_i$'s extend
uniquely to a  Borel measure $\mu_\infty$ on $X_\infty$.

\subsection{Measured Gromov-Hausdorff convergence}

In view of (\ref{eqn_distance_growth}), and since $\pi^\infty_i$ is also surjective,  
it follows easily that the sequence of mappings
$\{\pi^\infty_i:(X_\infty,d_\infty)\ra (X_i,d_i)\}$ is Gromov-Hausdorff convergent;
in particular the  Gromov-Hausdorff limit is isometric to $(X_\infty,d_\infty)$. 
By bringing in Axiom (5), we get that the sequence
$\{\pi^\infty_i:(X_\infty,d_\infty,\mu_\infty)\ra (X_i,d_i,\mu_i)\}$
is convergent in the pointed measured Gromov-Hausdorff sense;  for the definition, 
see \cite{fukaya}.
 Hence, we
obtain:
\begin{proposition}
\label{mghc}
 The sequence $(X_i,d_i,\mu_i)$  converges in the pointed 
measured Gromov-Hausdorff sense to $(X_\infty,d_\infty,\mu_\infty)$.  
\end{proposition}

\section{Bounded local geometry and verification of doubling} 
\label{s:blg}

Consider an admissible inverse system as in (\ref{is}), with constants, 
$2\leq m\in \N$, $\De$, $\th$, $C\in (0,\infty)$ as in (1)--(6). 
The following lemma asserts the existence of a local doubling condition, and 
a local Poincar\'e inequality. The proof is completely standard.
\begin{lemma}
\label{lpi}
 For all $K>0$, there exists $\beta'=\beta'(m,\De,\th,C,K)$, $\tau=\tau(m,\De,\th,C,K)$, 
$\Lambda(m,\De,\th,C,K)$,
such that for 
balls $B_r(x_i)\subset X_i$, with
$$
r\leq K\cdot m^{-i}\,  ,
$$
a doubling condition and (1,1)-Poincar\'e inequality hold , with constants 
$\beta'=\beta'(m,\De,\th,C,K ), \tau=\tau(m,\De,\th,C,K)$, $\Lambda=2$.
\end{lemma}

Next we verify the doubling condition for balls of arbitrary radius.

\begin{lemma}
\label{gdoubling}
There is a constant $\beta=\beta(\Delta,\theta,C,R)$ such that for all $i$
and all $r\leq R$, the doubling condition holds for $X_i$ with constant $\beta$.
\end{lemma}
\begin{proof}
 First, observe that since  for all $k$, from (\ref{containment}) and by
Axiom (5), $(\pi_k)_*(\mu_{k+1})=\mu_k$,  we get for $x_{k+1}\in X_{k+1}$,
\begin{equation}
\label{containment1}
\mu_{k}(B_r(\pi_{k}(x_{k+1})))\leq \mu_{k+1}(\pi_{j}^{-1}(B_{r+\theta m^{-(k+1)}}(x_{k+1})))\, ,
\end{equation}
\begin{equation}
\label{containment2}
\mu_{k+1}(\pi_{k}^{-1}(B_{s}(x_k)))
= \mu_k(B_{s}(x_k))\, .
\end{equation}

First assume that $R=1$.
Let $j$ be such that $m^{-(j+1)}<\frac{r}{1+2\theta}\leq m^{-j}$.  
Let $x_i\in X_i$ and consider $B_r(x_i)$. If $j\geq i$,
the conclusion follows from 
from Lemma \ref{lpi}. Otherwise, 
for $j+1\leq k\leq i$ inductively
define  $x_{k-1}=\pi_{k-1}(x_k)$.  
Since, $m^{-(j+1)}\leq \frac{r}{1+2\theta}$
 by (\ref{containment1}),
 (\ref{containment2})
 and induction we get
\begin{equation}
\label{containment3}
\mu_j(B_{\frac{r}{1+2\theta}}(x_j))\leq \mu_i(B_r(x_i)))
\leq \mu_i(B_{2r}(x_i))\leq \mu_j(B_{2r}(x_j))\, ,
\end{equation}
while  by (\ref{containment2}), we have
\begin{equation}
\label{containment4}
\mu_i(B_{2r}(x_i))\leq \mu_j(B_{2r}(x_j))\, .
\end{equation}
Since $x_j\in X_j$ and $\frac{r}{1+2\theta}\leq m^{-j}$, the conclusion follows
from (\ref{containment3}),  (\ref{containment4}) and
Lemma \ref{lpi}.

Now if $R>1$, the doubling inequality with $\be=\be(R)$
is equivalent to a doubling inequality 
for the graph $X_0$, which follows from the fact that it has controlled degree.\end{proof}


\section{Continuous fuzzy sections}
\label{spi}

Let $\mathcal P(Z)$ denote the space of Borel probability measures on $Z$ 
with the weak topology.

\begin{definition}
\label{fs}
Given a map of metric spaces $\pi:X\to Y$,  a {\bf fuzzy section of} $\pi $ is a 
Borel measurable map
from $\D:X\to \mathcal P(Y)$  such that $\D(x)$ is supported on 
$\pi^{-1}(x)$, for all $x\in X$.  $\D$ is called a {\bf continuous fuzzy section}
 if it is continuous with
 respect to the metric topology
on $X$ and the weak topology of $\P(Y)$.  The fuzzy sections in this paper are all
atomic, i.e. $\D(x)$ is a finite convex combination of Dirac masses.
\end{definition}

Here,  we will  observe that given
an admissible inverse system $\{(X_i,d_i,\mu_i,\pi_i)\}$ as in (\ref{is}),
each of the maps $\pi_i:X_{i+1}\to X_i$ has
a naturally associated continuous fuzzy section $\D_i$ defined via the measures $\mu_i,\mu_{i+1}$,
which satisfies
for some
$c_0>0$,
\begin{equation}
\label{lb1}
 \D_i(x_i)(x_{i+1})\geq c_0\qquad (\text{for all}\,\, i, \, x_i\in X_i, \, x_{i+1}\in\pi_i^{-1}(x_i))\, ,
\end{equation}
and has the additional property that if $e_{i+1}\subset X_{i+1}$ is an edge mapped isomorphically onto an edge $e_i\subset X_i$, then 
$x_i\mapsto\D_i(x_i)(e_{i+1})$
is constant
as $x_i$ varies in the interior of $e_i$;  see (\ref{fs2}).  This is used
 in Section \ref{s:piproof} in the proof of the Poincar\'e inquality.
We also observe that conversely, given
an inverse system of metric graphs $(X_i,d_i)$, as in (\ref{is}) which satisfies
(1)--(3), and  a sequence of
 continuous fuzzy sections $\D_i$ 
satisfying (\ref{lb1}), 
 there is a naturally associated sequence of measures
$\mu_i$ such that $\mu_0$ is normalized to be $1$-dimensional Lebesgue measure and
$(X_i,d_i,\mu_i)$ satisfies Axioms (1)--(6).  This reformulation is used
in Section \ref{eym}, in which of examples of admissible systems are constructed.    
 
Consider an admissible inverse system as in (\ref{is}). 
Let  ${\rm int}(e'_i)$ denote an open edge of $X_i'$, and ${\rm int}(e_{i+1})$ an open edge
of $X_{i+1}$, which is a component of $\pi_i^{-1}({\rm int}(e'_i))$.
  For $x_i\in {\rm int}(e'_i)$, $x_{i+1}\in \pi_i^{-1}(x_i)$ we define
 \begin{equation}
\label{fs1}
  \D_i(x_i)(x_{i+1})=\frac{\mu_{i+1}(e_{i+1})}{\mu_i(e'_i)}\, .
 \end{equation}
Thus, $\D_i$ is continuous on ${\rm int}(e'_i)$, and in fact, constant in  the sense that
for $x_{i,1}, x_{i,2}\in  {\rm int}(e'_i)$, $x_{i+1,1}\in e_{i+1}\cap\pi_i^{-1}(x_{i,1})$, 
\begin{equation}
\label{fs2}
\D_i(x_{i,1})(x_{i+1,1})=\D_i(x_{i,2})(x_{i+1,2})\, . 
\end{equation}

Next, suppose $v_i'$ is a vertex of $X_i'$ and $e_i'$ is an edge of $X_i'$ with $v_i'$ as
one of its end points.  If $v_{i+1}\in \pi_i^{-1}(v_i')$ then $v_{i+1}$ is a vertex of $X_{i+1}$ and
we define
\begin{equation}
\label{fs3}
\begin{aligned}
\D_i(v_i')(v_{i+1})& =\frac{\mu_{i+1}(\pi_i^{-1}(e_i')\cap \Star(v_{i+1},X_{i+1}))}{\mu_i(e_i')}\\
&=\sum_{e_{i+1}\in \Star(v_{i+1})}\frac{\mu_{i+1}(e_{i+1})}{\mu_i(e'_i)}\, .
\end{aligned}
\end{equation}
By (\ref{con1}) of  Axiom (6) (the continuity condition) $\D_i(x_i')(x_{i+1})$ is well
defined  independent of 
the choice of $e_i'$ with end point $v_i'$.
\begin{lemma}
\label{fsp}
$\D_i$ is a
continuous fuzzy section satisfying  (\ref{lb1}).
\end{lemma}
\begin{proof}

This follows immediately from (\ref{fs1}), (\ref{fs3}) that $\D_i$ is continuous

\begin{remark}
Note that $\D_i$ is simply the disintegration of $\mu_{i+1}$ with respect to the 
map $\pi_i:X_{i+1}\ra X_i$.
\end{remark}

 From (\ref{cb}), together with Axioms (3) and (4), it follows that $\D_i$ satisfies 
the lower bound (\ref{lb1}).
\end{proof}

The next proposition provides a sort of converse to the previous lemma.
\begin{proposition}
\label{equivalent}
Suppose the inverse system in (\ref{is}) satisfies (1)--(3).  Let
$\D_i$ denote a continuous fuzzy section of $\pi_i$, $i=0,1,\ldots$,
satisfying (\ref{lb1}) and (\ref{fs2}).
Let $\mu_0$ denote $1$-dimensional Lebesgue measure and define
$\mu_i$ inductively by (\ref{fs1}).  Then $\mu_i$ satisfies (4)--(6) for all
$i$.
\end{proposition}
\begin{proof}
Axiom (5) follows directly from the definition of $\mu_i$ via 
(\ref{fs1}) and the fact that $\D_i(x_i)$ is a probability measure for all $x_i$.
Axiom (6) follows directly from the assumption that the fuzzy section
$\D_i$ is continuous.  

  To verify Axiom (4), let $e_{i,1},e_{i,2}$ denote edges of $X_i$
with a common vertex $v_i$ of $X_i$. Define $v_k$ by downward 
induction, by setting $v_{k-1}=\pi_{k-1}(v_k)$. Let
$j\geq 0$ be either the largest value of $k$ such that $v_k$ is a vertex of $X_k'$ which
is not a vertex of $X_k$, or if there is no such $k$, put $j=0$. In either case, it is clear 
that $\mu_j(\pi_j\circ\cdots\pi_{i-1}(e_{i,1}))=\mu_j(\pi_j\circ\cdots\pi_{i-1}(e_{i,2}))$.

 From Proposition \ref{1-1} we get:
\vskip2mm
$(*)$
{\it For all but at most $\Delta$ values of $k$, the (locally surjective) map $\pi_{k-1}$ is $1$-$1$ in
a  neighborhood of $v_{k}$.} 
\vskip2mm

 Suppose, as in (*), 
the (locally surjective) map $\pi_{k}$ is $1$-$1$ in
a  neighborhood of $v_{k+1}$, and
$e_{k+1,1}$, $e_{k+1,2}$,
are edges with common vertex $v_{k+1}$.  Since $\D_k$ is continuous, 
by (\ref{fs1}),  we have
\begin{equation}
\label{equal}
\frac{\mu_{k+1}(e_{k+1,1})}{\mu_{k+1}(e_{k+1,2})}
=\frac{\mu_{k}(\pi_k(e_{k+1,1})}{\mu_{k}(\pi_k(e_{k+1,2}))}
\, .
\end{equation} 
 For the remaining values of $k$, by (\ref{lb1}),
\begin{equation}
\label{equal1}
c_0\leq 
\frac{\mu_{k+1}(e_{k+1,1})}{\mu_{k+1}(e_{k+1,2})}
\leq c_0^{-1}\, .
\end{equation} 
It follows that (4) holds with $C=(c_0)^\Delta$.
\end{proof}

\section{Proof of the Poincar\'e inequality and of Theorem
\ref{t:pi} }
\label{s:piproof}
 
In this section $i\geq 0$ will be fixed.

Given   $f_{i+1}:X_{i+1}\to \R$, we can perform integration of $f_{i+1}$ over the
fibers $\{\pi_i^{-1}(x_i)\}_{x_i\in X_i}$
of $\pi_i:X_{i+1}\ra X_i$ with respect to the family of measures $\{\D_i(x_i)\}_{x_i\in X_i}$, 
to 
produce a function on $X_i$ which we denote by $\I_{\D_i}f_{i+1}$.  Thus,
\begin{equation}
\label{flift}
\I_{\D_i}f_{i+1}(x_i)\defeq \sum_{x_{i+1}\in\pi_i^{-1}(x_i)}\D_i(x_i)(x_{i+1})f_{i+1}(x_{i+1})\, .
\end{equation}



By 
 (\ref{fs1}), (\ref{flift}), for all $A_{i}\subset X_i$, we have
\begin{equation}
\label{averages}
\int_{A_{i}}
\I_{\D_i}f_{i+1}\, d\mu_{i}=\int_{\pi_{i}^{-1}(A_{i})}f_{i+1}\, d\mu_{i+1}\, ;
\end{equation}
this also expresses the fact that $\D_i$ is the disintegration of $\mu_{i+1}$
with respect to $\pi_i$ and $\mu_i$ is the pushforward of $\mu_{i+1}$ by
$\pi_i$.

Now suppose $f_{i+1}$ is Lipschitz and let $\Lip\, f_{i+1}(x_{i+1})$ denote the pointwise Lipschitz
constant at $x_{i+1}\in X_{i+1}$.
Let $e_i'$ denote an edge of $X_i'$ and $e_{i+1}\subset \pi_i^{-1}(e_i')$ an edge of
$X_{i+1}$. Since by (\ref{fs2}), the function $\D_i(x_i)(x_{i+1})$ is constant as $x_i$ varies
in ${\rm int}(e_i')$ and $x_{i+1}$ varies in $\pi_i^{-1}(x_i)\cap {\rm int}(e_{i+1})$,
and since the restriction of $\pi_i$ to  $e_{i+1}$ is an isometry,  it follows
that that the restriction of $\I_{\D_i}f_{i+1}$ to ${\rm int}(e_i')$
 is Lipschitz, and

\begin{equation}
\label{pwl1}
\begin{aligned}
\Lip (\I_{\D_i}f_{i+1})(x_i)& \leq\sum_{x_{i+1}\in\pi_i^{-1}(x_i)}\D_i(x_i)(x_{i+1})\Lip\, f_{i+1}(x_{i+1})\\
                              &=\I_{\D_i}(\Lip\, f_{i+1})(x_i)\, .
\end{aligned}
\end{equation}

The following lemma depends crucially on the continuity assumption, Axiom (6) (as well
as on Axiom (4)); see also
(\ref{fs3}).
\begin{lemma}
\label{continuous}
If $f_{i+1}:X_{i+1}\to\R$ is Lipschitz then so is $\I_{\D_i}f_{i+1}$ and 
 for all $x_i\in X_i$ (including $x_i=v_i'$, a vertex of $X_i'$), we have
\begin{equation}
  \label{pwl}
\Lip (\I_{\D_i}f_{i+1})(x_i) \leq
                              \I_{\D_i}(\Lip\, f_{i+1})(x_i)\, .
\end{equation}
\end{lemma}
\begin{proof}
Clearly, it suffices to check that (\ref{pwl}) holds for $x_i=v_i'$ a vertex of $X_i'$.
Let  $v_i'$ a vertex of $e_i'$, $y_i\in{\rm int}(e_i')$ and
$v_{i+1}\in \pi_i^{-1}(v_i')$.  Then,
\begin{equation}
\label{con11}
\I_{\D_i}f_{i+1}(y_i)=\sum_{v_{i+1}\in\pi_i^{-1}(v_i')}\,\,\, 
\sum_{y_{i+1}\in\pi_i^{-1}(y_i)\cap \Star(v_{i+1},X_{i+1})}
\D_i(y_i)(y_{i+1})f_{i+1}(y_{i+1})\, .
\end{equation}
and  since the fuzzy section $\D_i$ is  {\bf continuous}, 
\begin{equation}
\label{con2}
\begin{aligned}
\I_{\D_i}f_{i+1}(v_i')&=\sum_{v_{i+1}\in\pi_i^{-1}(v_i')} \D_i(v_i)(v_{i+1})f_{i+1}(v_{i+1})\\
&=\sum_{v_{i+1}\in\pi_i^{-1}(v_i')}\,\,\, 
\sum_{y_{i+1}\in\pi_i^{-1}(y_i)\cap \Star(v_{i+1},X_{i+1})}\D_i(y_i)(y_{i+1})f_{i+1}(v_{i+1})\, .
\end{aligned}
\end{equation}
By subtracting 
 (\ref{con2}) from (\ref{con11}), dividing
through by $d_i(y_i,v_i')=d_{i+1}(y_{i+1},v_{i+1})$ and letting $y_i\to v_i'$, we easily obtain 
(\ref{pwl}).
\end{proof}

\begin{remark}
\label{ug}
We could as well have worked throughout with upper gradients. If
$g_{i+1}$ is an upper gradient for 
$f_{i+1}:X_{i+1}\to \R$, then a similar argument based on the continuity of 
$\D_i$ shows that $\I_{\D_i}g_{i+1}$ is
an upper gradient for $f_i=\I_{\D_i}f_{i+1}$.
\end{remark}

\begin{proposition}
\label{pupi}
Given an admissible  inverse system as in (\ref{is}),
 for all $i$ and $R$, a $(1,1)$-Poincar\'e inequality holds for balls 
$B_r(x_i)\subset X_i$, with $\tau=\tau(\delta,\theta,C)$ and
$\Lambda=2(1+\theta)$.
\end{proposition}
\begin{proof}
Without essential loss of generality,
it suffices to assume $R= 1$. Given $0<r\leq 1$, let $j$ be such that
$$
m^{-(j+1)}<r\leq m^{-j}\, .
$$
Let $B_r(x_i)\subset X_i$.  If $r\leq m^{-i}$ then Lemma \ref{lpi} applies.
Thus, we can assume $m^{-i}<r$.

 For $j+1\leq k< i$, inductively define 
\begin{equation}
\label{xjdef}
\!\!\!\!\!\!\!\!\!\!\!\!\!\!\!\!\!\!\!\!\!\!\!\!\!\!x_k=\pi_{k}\circ\cdots \circ \pi_{i-1}(x_i)\, ,
\end{equation}
\begin{equation}
\label{Ujdef}
\begin{aligned}
U_{j+1} &=B_r(x_{j+1})\, ,\\
U_k&= \pi_k^{-1}(U_{k-1})\qquad j+1\leq k< i\, .
\end{aligned}
\end{equation}
By (\ref{surj}), and induction, we have 
\begin{equation}
\label{contain}
B_r(x_i)\subset U_i\subset B_{(1+\theta)r}(x_i)\, .
\end{equation}



 Given a Lipschitz function $f_i:X_i\to\R$, 
set
\begin{equation}
\label{fkdef}
f_{k-1}=\I_{\D_{k-1}}f_{k}\qquad   j+1\leq k<i\, ,\\
\end{equation}
\begin{equation}
\label{fhatdef}
\hat f_k=f_{k-1}\circ\pi_{k-1}^{-1}\, .
\end{equation}
Then for all $A_{k-1}\subset X_{k-1}$ and $A_k\defeq \pi_{k-1}^{-1}(A_{k-1})$, we have
\begin{equation}
\label{av}
(f_k)_{A_k}=(f_{k-1})_{A_{k-1}}=(\hat f_k)_{A_k}\, .
\end{equation}
In particular, since $(\hat f_{i})_{U_i}=  ( f_{i-1})_{U_{i-1}} $,  we get
$$
\begin{aligned}
\int_{U_i}|f_{i}-(f_{i})_{U_i}|\, d\mu_i & \leq
\int_{U_i}|f_{i}-\hat f_{i}|\, d\mu_i +
                                                           \int_{U_i}|\hat f_{i}-(\hat f_{i})_{U_i}|\,d\mu_i\\
                                                       &= \int_{U_i}|f_{i}-\hat f_{i}|\, d\mu_i+
                                                           \int_{U_{i-1}}| f_{i-1}-( f_{i-1})_{U_{i-1}}|\, d\mu_{i-1}\, ,
\end{aligned}
$$
and by induction,
\begin{equation}
\label{sum}
\int_{U_i}|f_{i}-(f_{i})_{U_i}|\, d\mu_i\leq \sum^i_{k\geq j+2} \int_{U_k}|f_k-\hat f_k|\, d\mu_k 
 +\int_{B_r(x_{j+1})}|f_{j+1}-{(f_{j+1})}_{B_r(x_{j+1})}|\, d\mu_j\, .
\end{equation}

By (\ref{containment}) and induction, we have 
$$
U_i\subset B_{(1+\theta)r}(x_i)\, .
$$
Using Lemma \ref{lpi}, Lemma \ref{continuous}, (\ref{contain}) and induction, for 
$\tau=\tau(\Delta,\theta,C)$, 
the Poincar\'e inequality on $B_r(x_j)$ gives following estimate 
for the second term on the r.h.s of (\ref{sum}).
\begin{equation}
\label{lpibound}
\begin{aligned}
\int_{B_r(x_j)}| f_j-{( f_j)}_{B_r(x_j)}|\, d\mu_i
&\leq \tau r\cdot \int_{B_r(x_j)}\Lip\, f_j\, d\mu_j\\
&\leq \tau r\cdot \int_{U_i}\Lip\, f_i\, d\mu_i\\ 
&\leq \tau r\cdot\int_{B_{(1+\theta)r}(x_i)}\Lip\, f_i\, d\mu_i\, .
\end{aligned}
\end{equation}
 
Next we estimate the remaining terms on the r.h.s. of (\ref{sum}). For all $j+2\leq k\leq i$,
let $\{x_{k-1,t}\}$ denote  a maximal $m^{-k}$-separated subset of
$U_{k-1}$. 
It follows from the local doubling condition that the collection of balls,
$\{B_{m^{-k}}(x_{k,t})\}$ covers $U_k$ and has multiplicity bounded by
a constant $M(\beta)$, with $\beta$ the local doubling constant in Lemma \ref{lpi}.

Set $U_{i,k,t}=(\pi_k\circ\cdots\pi_{i-1})^{-1}(B_{(1+\theta)m^{-k}}(x_{k,t}))$.
By (\ref{av}), we have
$$
(f_k-\hat f_k)_{\pi_{k-1}^{-1}(B_{m^{-k}(x_{k-1,t})})}=0\, .
$$
 Thus, we get
$$
\begin{aligned}
&\int_{\pi_{k-1}^{-1}(B_{m^{-k}}(x_{k-1,t}))}|(f_k-\hat f_k)|\, d\mu_k\\
&{}\,\,\,\,\,\,\,=\int_{\pi_{k-1}^{-1}(B_{m^{-k}}(x_{k-1,t}))}|(f_k-\hat f_k)-(f_k-\hat f_k)_{B_{m^{-k}}(x_{k-1,t})} |\, d\mu_k\\
&{}\,\,\,\,\,\,\,\leq \int_{\pi_{k-1}^{-1}(B_{m^{-k}}(x_{k-1,t}))}|(f_k-\hat f_k)-(f_k-\hat f_k)_{B_{(1+\theta)m^{-k}}(x_{k-1,t})} |\, d\mu_k\\
&{}\,\,\,\,\,\,\,\leq  2\int_{B_{(1+\theta)m^{-k}}(x_{k,t}) }    |(f_k-\hat f_k)-(f_k-\hat f_k)_{B_{(1+\theta)m^{-k}}(x_{k,t})} |\, d\mu_k\\
&{}\,\,\,\,\,\,\,\leq 4\tau(1+\theta)m^{-k}\cdot\int_{B_{(1+\theta)m^{-k}}(x_{k,t}) } 
\Lip\, f_k\, d\mu_k\\
&{}\,\,\,\,\,\,\,\leq 4\tau(1+\theta)m^{-k}\cdot\int_{U_{i,k,t} } 
\Lip\, f_k\, d\mu_k\\
\end{aligned}
$$
where the penultimate inquality comes from using $\Lip\, (f_k-\hat f_k)\leq 2\Lip\, f_k$
and
applying the Poincar\'e inquality on 
$B_{(1+\theta)m^{-k}}(x_{k,t})$. 
By summing this estimate over $t$  and $k$, and using $\bigcup_t U_{i,k,t}
\subset B_{2(1+\theta)r}(x_i)$,
the proof is completed.
\end{proof}


\vskip1mm
\begin{proof}[Proof of Theorem \ref{t:pi}]
We have observed in Proposition \ref{mghc} that $\{(X_n,d_n,\mu_n)\}$ converges to 
$(X_\infty,d_\infty,\mu_\infty)$
in the measured Gromov-Hausdorff sense. Since the doubling condition
and Poincar\'e inequality with uniform constants pass to measured
Gromov-Hausdorff limits \cite{cheeger}, \cite{keith},  the theorem follows from
Propositions \ref{gdoubling}, \ref{pupi}.
\end{proof}



\section{A probability measure on the lifts of a path}
\label{mlp}
In this section we define a probability measure 
$\Omega$ on the set of lifts to $X_i$ ($i>k$)
of a path $\gamma_k$ in $X_k$ and establish a particular property
which is a consequence of Axiom (6); see Proposition
\ref{pfml}.
This  property  plays a role in Section \ref{s:p2},
in which we give an alternative proof of the Poincar\'e inequality. 
The measure $\Omega$ has an interpretation in
terms of Markov chains which is explained in Remark \ref{markov} at the end of the 
section;  it  
also enters  
in Section \ref{eym}, in which we construct examples of admissible inverse
systems. We begin with the case $i=k+1$ from which the general case follows
easily.

A {\bf vertex path} in $X_k'$ is a sequence of vertices $v_{0,k}',\ldots,v_{N+1,k}'$ such that 
each pair of consecutive vertices  are the vertices of an edge 
of $X_k'$.
Associated to a vertex path is the {\bf path} $\gamma_{k}'=e_{0,k}'\cup\cdots\cup e_{N,k}'$,
which we will always assume is parameterized by arclength.
Similarly, we define a {\bf path} $\gamma_{k+1}=e_{0,k+1}\cup\cdots\cup e_{N,k+1}$ in
$X_{k+1}$ associated to $v_{0,k+1},\ldots,v_{N+1,k+1}$.
We denote by $\Gamma$, the (finite) collection of all $\gamma_{k+1}$ that are lifts of 
$\gamma_k'$.

Below, given $e_k'$ and a lift $e_{k+1}$, by slight abuse of notation (compare (\ref{fs1})) we write
\begin{equation}
\label{sa}
\D_k(e_k')(e_{k+1})\defeq\frac{\mu_{k+1}(e_{k+1})}{\mu_k(e_k')}\, .
\end{equation}

Define a measure $\Omega$ on $\Gamma$ by setting
\begin{equation}
\label{measdef1}
\Omega(\gamma_{k+1}) \defeq\D_k(e_{0,k}')(e_{0,k+1})\times
\left(\frac{\D_k(e_{1,k}')(e_{1,k+1})}{\D_k(v_{1,k}')(v_{1,k+1})}\right)\times\cdots\times
\left(\frac{\D_k(e_{N,k}')(e_{n,k+1})}{\D_k(v_{N,k}')(v_{N,k+1})}\right)\, ,
\end{equation}
where by (\ref{fs3}), we can write
\begin{equation}
\label{fs3'}
\D_k(v_{j,k}')(v_{j,k+1}) 
=\sum_{\underline{e}_{j,k+1}\in \pi_k^{-1}(e_{j,k}')\cap \Star(v_{j,k+1})}\frac{\mu_{k+1}(\underline{e}_{j,k+1})}{\mu_k(e'_{j,k})}\, .
\end{equation}

For a path, $\gamma_k'=e_{0,k}'$, consisting of a single edge, and a lift, 
$\gamma_{k+1}=e_{0,k+1}$, we just have 
\begin{equation}
\label{ge}
\Omega(e_{0,k+1})=\D_k(e_{0,k}')(e_{0,k+1})\, .
\end{equation} 
Since $\D_k(x_{0,,k}')(\,\cdot\,)$ is a probability measure,  it follows directly from 
the definitions
that $\Omega$ is a probability measure in this case.

We now check an important  property of 
$\Omega$ which in particular, implies that $\Omega$ is a probability measure for arbitrary
$\gamma_k'$; see (\ref{equal2}). Let $\psi_k'$ denote a path consisting of $N+1$ edges obtained from $\gamma_k'$ by 
adjoining
a single edge $e_{N+1,k}'$. Let $\Psi$ denote the collection of all 
lifts of $\psi'_k$ and let $\Omega_{\psi'_{k+1}}$ denote the measure 
on $\Psi$ (defined as in (\ref{measdef1})).
$\overline{\Psi}$ denote the collection of lifts of 
$\psi_k'$ containing the {\it fixed} lift $\gamma_{k+1}$
of $\gamma_k'$.
Then it follows from (\ref{sa}) and (\ref{measdef1}), together with
 (\ref{fs3'}) applied to the vertices $v_{N+1,k}'$, 
$v_{N+1,k+1}$,
that
\begin{equation}
\label{equal2}
\Omega_{\psi'_{k+1}}(\ol{\Psi})=\Omega(\gamma_{k+1})\, .
\end{equation}
It now follows by induction that $\Omega$ is a probability measure for arbitrary 
$\gamma_k'$; compare Remark \ref{markov}.

\begin{remark}
Note that  if we understand (\ref{fs3'}) to be the definition $\D_k(v_{j,k}')(v_{j,k+1})$
then the discussion to this point has not made use of Axiom (6).  
\end{remark}

Recall that Axiom (6) implies that $\D_k(v_{j,k}')(v_{j,k+1})$ depends only on $v_{j,k}',v_{j,k+1}$,
and in particular (compare (\ref{fs3'})) we also have
\begin{equation}
\label{fs3''}
\D_k(v_{j,k}')(v_{j,k+1}) 
=\sum_{\underline{e}_{j-1,k+1}\in \pi_k^{-1}(e_{j-1,k}')\cap \Star(v_{j,k+1})}\frac{\mu_{k+1}(\underline{e}_{j-1,k+1})}{\mu_k(e'_{j-1,k})}\, .
\end{equation}
If we rewrite the expression in (\ref{measdef1}) for $\Omega$ as
\begin{equation}
\label{measdef}
\Omega(\gamma_{k+1}) =
\frac{\D_k(e_{0,k}')(e_{0,k+1})\times \cdots \times \D_k(e_{N,k}')(e_{N,k+1})}{\D_k(v_{1,k}')
(v_{1,k+1})\times\cdots\times\D_k(v_{N,k}')(v_{N,k+1})}\, , 
\end{equation}
 we easily obtain:

 \begin{proposition}
\label{meassym} For an admissible inverse system,  the measure $\Omega$ is invariant under
the operation of reversing the orientations of $\gamma_k'$, $\gamma_{k+1}$.
\end{proposition}

It follows immediately from Proposition \ref{meassym}, that (\ref{equal2}) also
 holds if the additional edge is adjoined
at the begining  of $\gamma_k'$ rather than at the end. From  this  and an
argument  by induction, we get the
following:
For arbitrary $\gamma_k'$, if $\psi_k'$ is any path containing $\gamma_k'$,
$\gamma_{k+1}$ is any fixed lift of $\gamma_k'$ and $\ol{\Psi}$ denotes the collection
of all lifts of $\psi_k'$ containing $\gamma_{k+1}$ then  (\ref{equal2}) holds.  This gives:
\begin{corollary}
If $e_{j,k}'$ is any edge 
contained in $\gamma_k'$, $e_{j,k+1}\in \pi^{-1}(e_{j,k}')$ and $\overline{\Gamma}$
denotes the collection of lifts of $\gamma_k'$ which contain $e_{j,k+1}$, then
\begin{equation}
\label{edge}
\Omega(\overline{\Gamma})=\D_k(e_{j,k}')(e_{j,k+1})
=\frac{\mu_{k+1}(e_{j,k+1})}{\mu_k(e'_{j,k})}\, .
\end{equation}
\end{corollary}

Next, we give a consequence of (\ref{edge}) which is used in the 
alternate proof of the Poincar\'e inequality given in
Section \ref{s:p2}.

Suppose that $\gamma_k'$ 
is the subdivision of a  path in $X_k$ consisting of the union of
$L$ edges $e_{0,k}\cup\cdots\cup e_{L,k}$
 of $X_k$. (Thus, $\gamma_k'$ has $L\cdot m$ edges $e_{j,k}'$.)
Assume that $\gamma_k'$ is parameterized by arclength.
Define \hbox{$\Phi:\Gamma\times [0,L\cdot m^{-k}]\to X_k$} by
$$
 \Phi(\gamma_{k+1},t)=\gamma_{k+1}(t)
$$
Let $\L$ denote Lebesgue measure on $X_{k+1}$.

We claim that on any fixed $e_{\ell,k}$ in the domain of $\gamma_{k}$, we have
$$
\Phi_*(\Omega\times \L)= \frac{m^{-k}}{\mu_{k+1}(\pi_k^{-1}(e_{\ell,k}))}\cdot \mu_{k+1}\, ,
$$
where $\Phi_*$ denotes push forward under the map $\Phi$.
To see it, note that for any $e_{j,k+1}$ we have
$$
\L=\mu_{k+1}\cdot\frac{m^{-(k+1)}}{\mu_{k+1}(e_{j,k+1})}\, ,
$$
If  $e_{j,k}'\subset e_{\ell,k}$ and $e_{j,k+1}\subset  \pi_k^{-1}(e_{j,k}')$, then
on $e_{j,k+1}$ we have
by (\ref{edge})
$$
\Phi_*(\Omega\times \L)=\frac{\mu_{k+1}(e_{j,k+1})}{\mu_k(e_{j,k}')}\cdot\L\, .
$$
Combining the previous two relations gives
\begin{equation}
\label{pf}
\begin{aligned}
\Phi_*(\Omega\times \L)&=\frac{m^{-(k+1)}}{\mu_k(e'_{j,k})}\cdot \mu_{k+1}\\
&=\frac{m^{-k}}{\mu_{k+1}(\pi_k^{-1}(e_{\ell,k}))}\cdot \mu_{k+1}\, ,
\end{aligned}
\end{equation}
 where the last equality follows by because $\mu_k$ is a constant multiple of Lebesgue measure
on $e_{\ell,k}$  and ${(\pi_k)}_*(\mu_{k+1})=\mu_k$.

Finally, we give a generalization of the above. Put
 $\pi_k^i=\pi_{k}\circ \cdots \circ \pi_{i-1}$.
Write $X_k^i$ for $X_k$ with each of its edges subdivided into 
edges of length $m^{-(i-1)}$. Then $\pi_k^i$ is maps edges of $X^i$
to edges of $(X_k^i)'$ It is easy to see that after  rescaling
of the metric and measure on both $X_k^i$ and $X_i$ by a factor $m^{i-1}$,
Axioms (1)--(6) are satisfied  (where the verification
of Axiom (6) is by induction).  In addition, the $X_k^i$ with rescaled
metric has  the  property that the rescaled $\mu_i$ is a
constant multiple of $\L$ on the edges of the rescaled $X_k$ (which have length
$m^{i-k-1}$ in the rescaled metric).  As a consequence,
by the same argument which led to  (\ref{pf}), we get:
\begin{proposition}
\label{pfml}
Let $\gamma_k$ denote a path in $X_k$ which is the union of edges $e_k$ 
of $X_k$ and let $\gamma_k^i$ denote its
subdivision in $X_k^i$. If  $\Gamma $ denotes collection 
of lifts of $\gamma_k^i\subset X_k^i$ to $X_i$, then there is a  probability 
measure $\Omega$ on $\Gamma$ such that 
\begin{equation}
\label{pff}
\Phi_*(\Omega\times \L)
=\frac{m^{-k}}{\mu_{i}((\pi_k^i)^{-1}(e_{\ell,k}))}\cdot 
\mu_{i}\,\qquad ({\rm on}\,\,e_{\ell,k} ).
\end{equation}
\end{proposition}

\begin{remark}
\label{markov}
 The definition of $\Omega$ in (\ref{measdef})
can be understood in terms of Markov chains. This gives a more general perspective 
on why it is a probability measure.
Associated to $\gamma_{k+1}'$ is a discrete time Markov chain whose
collection of states is $\bigcup_{j=0}^N (\pi_k^{-1}(e_{k,j}'),j)$.
The probability of being in a state $(e_{j,k+1},j)$ at time $0$ is $0$ unless
$j=0$,
in which case the probability is $\D(e_{0,k})(e_{0,k+1})$.
The probability of transition from a state $(e_{j_1,k+1},j_1)$ at time $j$
to a state $(e_{j_2,k+1},j_2)$ at time $j+1$ is $0$ unless $j_1=j,\,\, j_2=j+1$
and there exists $\gamma_{k+1} \in \Gamma$ such that
$e_{j,k+1}, \, e_{j+1,k+1}$ are consecutive edges of $\gamma_{k+1}$ with common
vertex $v_{j+1,k+1}$, and such that $e_{j_1,k+1}=e_{j,k+1}$ and
$e_{j_2,k+1}=e_{j+1,k+1}$.
In this case the transition probability is
$$
\frac{\D(e_{j+1,k}')(e_{j+1,k+1})}{\D(v_{j+1,k}')(e_{j+1,k+1})}\defeq
\frac{\mu_{k+1}(e_{j,k+1})}{\sum_{\underline{e}_{j,k+1}\in\pi_k^{-1}(e_{j,k}')\cap \Star(v_{j,k+1})}
\mu_{k+1}(\underline{e}_{j,k+1})}\, ;$$

For this Markov chain, the probability of observing a sequence of states
$(e_{j_0,k+1},0), (e_{j_1,k+1},1), \dots, (e_{j_N,k+1},N)$
is zero unless there exists $\gamma_{k+1}=e_{0,k+1}\cup\cdots\cup e_{N,k+1}
\in \Gamma$, with $e_{j_0,k+1}=e_{0,k+1},\ldots, e_{j_N,k+1}=e_{N,k+1}$,
in which case this probability is $\Omega(\gamma_{k+1})$.

Note that the in above discussion we need not assume that 
Axiom (6) holds. However, this assumption is required for
 Proposition \ref{meassym} whose consequence,
Proposition \ref{pfml}, is crucial for the alternate proof of the Poincar\'e inequality given in the next section.
\end{remark}

\section{A proof of the Poincar\'e inequality using measured
 path families}
\label{s:p2}

In this section we give an second proof based on measured path families that
the Poincar\'e inequality holds for $(X_\infty,d_\infty,\mu_\infty)$.\footnote{ As a matter of convenience,
some of the notational conventions of this section are
somewhat  at variance with those of other sections and (given that this is 
our second proof of the Poincar\'e inequality) the style of presentation is
slightly more informal.}
This is closer in spirit to other proofs of the Poincar\'e
inequality \cite{semmes}.

Suppose $k\leq i$, $v_k$ is a vertex of $X_k$, $e_{0,k}, e_{1,k}$ are edges belonging to 
the star of $v_k$ in $X_k$,  and $Z_\ell=(\pi_k^i)^{-1}(e_{\ell,k})\subset X_i$ for $\ell\in \{0,1\}$.
Let 
$\gamma_k:[0,2m^{-k}]\ra X_k^i$ 
denote a unit speed parametrization of the path
$e_{0,k}\cup e_{1,k}$ and $\gamma_k^i$ its subdivision in $X^i_k$.
 Let $\Ga$ denote the 
space of lifts $\gamma_i:[0,2m^{-k}]\ra X_i$ of 
$\gamma_k^i$
 and let $\Omega$ denote the probability measure on $\Gamma$ constructed
in Section \ref{mlp}.
Let $\Phi:\Gamma\times [0,2m^{-k}]\ra Z_0\cup Z_1\subset X_i$ denote the tautological map
$(s,\gamma_{i})\mapsto \gamma_{i}(s)$.

Recall from (\ref{igdf}) the definition of an upper gradient $g$ of a function 
$f$ on a metric space.

\begin{lemma}
\label{cor-adjacent_average}
Let $k<  i,\,Z_0,Z_1$ are as above.
Let  $u:X_i\ra \R$ denote a Lipschitz function
and $g:X_i\ra \R$ an upper gradient for $u$.  Then 
$$
\left|\av_{Z_0}u\,d\mu_i -\av_{Z_1}u\,d\mu_i \right|
\leq \hat Cm^{-k}\av_{Z_0\cup Z_1}g\,d\mu_i\,.
$$
\end{lemma}
\begin{proof}
With Axiom (4) and (\ref{pff}) of Proposition \ref{pfml} (which is used twice below) we get:
\begin{align*}
\left|\av_{Z_0}u\,d\mu_i \right. &- \left.\av_{Z_1}u\,d\mu_i \right|\\
&=\left|   \av_{[0,m^{-k}]\times \Ga}(u\circ \Phi) \;d(\L\times\Omega)
-   \av_{[m^{-k},2m^{-k}]\times \Ga}(u\circ \Phi) \;d(\L\times\Omega)\right|\\
&\leq \av_{[0,m^{-k}]\times\Ga}  \left| u(\gamma_{i}(t))-u(\gamma_{i}(t+m^{-k}))\right|2
\;d\L(t)\;d\Omega(\eta)  \\
&\leq \av_{[0,m^{-k}]\times\Ga}\int_{[0,m^{-k}]}
g\circ\gamma_{i}(t+s)\;d\L(s)\;d\L(t)\;d\Omega(\hat\eta)\\
&=\int_{[0,m^{-k}]}\left(\av_{[0,m^{-k}]\times\Ga}
g\circ\gamma_{i}(t+s)\;d\L(t)\;d\Omega(\gamma_i)\right)\;d\L(s)\\
&\leq\hat C \int_{[0,m^{-k}]}\left( \av_{Z_0\cup Z_1}g\;d\mu_i   \right)\;d\L(s)\\
&=\hat C\; m^{-k}\av_{Z_0\cup Z_1}g\;d\mu_i\,.
\end{align*}
\end{proof}

\begin{theorem}
$(X_\infty,d_\infty,\mu_\infty)$ satisfies a Poincar\'e inequality.
\end{theorem}
\begin{proof}
It suffices to prove that $(X_i,d_i,\mu_i)$ satisfies a Poincar\'e inequality for every $i\in \Z$, with
constant indendent of $i$; see \cite{cheeger}, \cite{keith}.   We fix $i\in \Z$, and let $u:X_i\ra \R$ denote a Lipschitz function with
upper gradient $g:X_i\ra \R$.  For every  $k\leq i$, let $\U_k^i$ denote the collection of subsets of $X_i$
of the form $U_k^i=(\pi_k^i)^{-1}(e_k)$, where $e_k$ is an edge of $X_k$.   
Let $u_{i,k}:X_i\ra \R$ denote
a step function such that for every $U_k^i\in \U_k^i$,
$$
u_{i,k}(x_i)=\av_{U_k^i} u\;d\mu_i\, ,
$$
for  $\mu_i$-a.e. $x_i\in U_k^i$. 
In particular, $u_{i,i}$ satisfies
$$
u_{i,i}(x_i)=\av_{e_i} u\;d\mu_i\, ,
$$
for  all edges $e_i$ of $X_i$ and $\mu_i$-a.e. $x_i\in e_i$.

Let $k<i$, and $U_k^i=(\pi_k^i)^{-1}(e_k)\in \U_k^i$.   If two elements 
$U_{0,k+1}^i=(\pi_{k+1}^i)^{-1}(e_{0,k+1})$, 
$U_{1,k+1}^i=(\pi_{k+1}^i)^{-1}(e_{1,k+1})\in \U_{k+1}^i$ 
are  contained in some $U_k$, then by Axiom (3) (the diameter bound on fibres)
 $e_{0,k+1},e_{1,k+1}$ are at distance $\leq C=C(\theta)m^{-k}$
 in $X_{k+1}$, and so by
Lemma \ref{cor-adjacent_average} and induction, we have
$$
\left|\av_{U_{0,k+1}^i}u\;d\mu_i-\av_{U_{1,k+1}^i}u\;d\mu_i\right|
\leq \hat C\cdot m^{-k}\av_{CU_k^i}g\;d\mu_i\,,
$$ 
where $CU_k^i$ denotes 
of a tubular neighborhood of radius $C(\theta) m^{-k}$ around
$e_k$;  see (\ref{containment}).   

 Since  at most a definite number of elements of $\U_{k+1}^i$ are contained in a fixed $U_k^i$
(see (\ref{cb}))
 this gives for all $k\leq i-1$,
\begin{equation}
\label{eqn-j_to_j+1}
\int_{U_j^i}|u_{i,k}-u_{i,k+1}|\;d\mu_i\leq C_1 m^{-k} \int_{CU_{k+1}^i}g\;d\mu_k\,.
\end{equation}
where $C_1=C_1(m,\Delta, \theta)$.

Now suppose $j\leq i$, $v_j$ is a vertex of $X_j$, 
and let $Z=(\pi_j^i)^{-1}(\Star(v_j,X_j))\subset X_i$.
By (\ref{eqn-j_to_j+1}) (with notation as above) we have 
\begin{equation}
\begin{aligned}
\int_Z&|u_{i,i}-u_{i,j}|\;d\mu_i \leq \sum_{k=j}^{i-1}\int_Z|u_{k,j+1}-u_{k,j}|\;d\mu_k\\
\label{eqn-telescope}
&\leq \sum_{k=j}^{i-1}C_1m^{-j}\int_{CZ}g\;d\mu_i
\leq C_1m^{-j}\int_{CZ}g\;d\mu_i\, .
\end{aligned}
\end{equation}

Applying the Poincar\'e inequality for each edge $e_i$ of $Z$ gives
\begin{equation}
\label{eqn-pi_on_edge}
\int_{Z}|u-u_{i,i}|\;d\mu_i\leq m^{-i}\int_{Z}g\;d\mu_i\, .
\end{equation}
Since $X_j$ has a valence bound independent of $j$, it follows from
 Lemma \ref{cor-adjacent_average} that
\begin{equation}
\label{eqn-u_i_to_u_y}
\int_Z|u_{i,j}-u_Z|\;d\mu_i\leq \hat{C} m^{-j}\int_{Z}g\;d\mu_i\,.
\end{equation}

Combining (\ref{eqn-telescope}), (\ref{eqn-pi_on_edge}), and (\ref{eqn-u_i_to_u_y}) we obtain
\begin{equation}
\begin{aligned}
\int_Z|u-u_Z|\;d\mu_i&\leq \int_Z (|u-u_{i,i}|+|u_{i,i}-u_{i,j}|+|u_{i,j}-u_Z|)\;d\mu_i\\
&\leq  Cm^{-j}\cdot \int_{CZ}g\;d\mu_i\,.
\end{aligned}
\end{equation}

 Since $X_i$ has valence bounded independent of 
$i$ and edges of length $m^{-i}$,
it suffices to prove the Poincar\'e inequality for balls $B_r(x_i)$ where $r$ is at least 
comparable
 to $m^{-i}$, 
since otherwise $B_r(x_i)$ lies in the star of some vertex $v_i\in X_i$, and the result is trivial; 
see Lemma \ref{lpi}.   Thus, we may
assume that there is a $j\leq k$ with $m^{-j}$ comparable to $r$ and a vertex $v_j\in X_j$ such that 
$\pi_j^i(B_r(x_i))\subset \Star(v_i,X_i)$.
Letting $Z=(\pi_j^i)^{-1}(\Star(v_j,X_j))$, we have $B_r(x_i)\subset Z$ and 
$\mu_i(Z)/\mu_k(B_r(x_i)$ has a definite
bound; see Axiom (4).  
Then
$$
\av_{B_r(x_i)} |u-u_{B_r(x_j)}|\;d\mu_i\leq C\; \av_Z|u-u_Z|\;d\mu_i\leq C\;m^{-j}\av_{CZ}g\;d\mu_i
$$
$$
\leq C\;m^{-j}\av_{B_{Cr}(x_i)}g\;d\mu_i\,.
$$
This suffices to complete the proof.
\end{proof}

\section{Construction of admissible inverse systems}
\label{eym}

In view of Theorem \ref{t:pi}, it is natural to ask for explicit examples of
admissible inverse systems  and whether (and in what sense) it is possible to classify them.
In this section we will  content ourselves with giving an inductive procedure for 
constructing admissible inverse systems,
which makes it clear that combinatorially distinct admissible 
inverse systems exist in great abundance. We will also give a simple example of an
 inverse system of metric graphs satisfying Axioms (1)--(3) which cannot be given
the structure of an admissible inverse system, i.e. for this inverse system,
a sequence of measures $\mu_k$, satisfying Axioms (4)--(6) does not exist; see Example \ref{nontriv}.

\subsection{Admissible edge inverses; the simplest special
case}
\label{subsec_inductive}

Given an admissible inverse system $\{X_i\}_{i\in \Z_+}$, one may think of
$X_{k+1}$ as the union the  subgraphs $\pi_k^{-1}(e_k)$, where $e_k\subset X_k$
ranges over all edges of $X_k$.  The following definition axiomatizes the
properties of these  subgraphs, up  to rescaling of the metric and the measure.

\begin{definition}
\label{def_admissible_edge_inverse}
An {\bf admissible edge inverse} is a map 
$(Y_1,d_1,\nu_1)\stackrel{\pi}{\lra}(Y_0,d_0,\nu_0)$
of finite metric measure graphs, satisfying the following conditions for some  
integer $m\geq 2$:

\begin{enumerate}
\renewcommand{\theenumi}{\Alph{enumi}}
\setlength{\itemsep}{1ex}

\item $(Y_0,d_0,\nu_0)$ is a copy of  the unit interval
$[0,1]$ with the usual metric and measure.
$Y_1$ is a nonempty, finite, possibly disconnected 
graph, such that every edge $e_1\subset Y_1$ is isometric to an interval 
of length $\frac{1}{m}$.
The restriction of $d_1$ to every component of $Y_1$ is the associated
path metric.
The restriction of the measure $\nu_1$ to $e_1$ is a nonzero multiple of the arclength.

\item If $Y_0'$ denotes the result of subdividing $Y_0\simeq [0,1]$ into $m$ edges
of length $\frac{1}{m}$, then
$\pi:Y_1\ra Y_0'$ is open, and its restriction to any edge $e_1\subset Y_1$
maps $e$ isometrically onto an edge of $Y_0'$.

\item   (Compatibility with projections)\;  The pushforward  $\pi_*(\nu_1)$
is $\nu_0$.
\item (Continuity)\;  For every vertex $v\in Y_0'$, and every $w\in \pi^{-1}(v)\subset Y_1$, 
the quantity 
$$
\frac{\nu_1(\pi^{-1}(e_0)\cap \Star(w,Y_1))}{\nu_0(e_0)}
$$
is the same for all edges $e\subset \Star(v,Y_0')$. 
\end{enumerate}
\end{definition}

Note that if $\{X_i\}_{i\geq 0}$ is an admissible inverse system with subdivision 
parameter $m$, then for any $i$ and any edge $e\subset X_i$, the restriction of $\pi_i$
to $\pi_i^{-1}(e)$ yields an admissible edge inverse $\pi_i:\pi_i^{-1}(e)\ra e$,
modulo rescaling the metric and normalizing the measure.

Fix $m,n\geq 2$, and an admissible edge inverse
$\pi:(Y_1,\nu_1)\ra (Y_0,\nu_0)$
 with subdivision parameter $m$. We now assume further
that if $v\in\{0,1\}$ is an endpoint of $Y_0\simeq [0,1]$
 then $\pi^{-1}(v)$ has cardinality $n$. For
each such end point, choose and identification of
the set of inverse images with the set $\{1,\ldots,n\}$. Moreover, assume that
\begin{align}
 \label{eqn_connected}
&\text{$Y_1$ is connected and $d_1$ is a length metric on $Y_1$.}\\
\notag&\\
\notag&\text{If $v\in \{0,1\}$ is an endpoint of $Y_0\simeq [0,1]$  and $w\in \pi^{-1}(v)$,} \\
\label{eqn_degree_1_equal_measure}&
\text{then $w$ has degree $1$, and the unique edge containing $w$ has}\\
&\notag\text{$\nu_1$ measure $\frac{1}{mn}$.}
\end{align}

\subsection{Inductive construction of admissible inverse systems}

Fix $m$ and
$N<\infty$ and assume that for each integer $n$ with $1\leq n\leq  N$ we have a 
finite nonempty
family $\G(n)$ of edge inverses as above
as above such that  for $v$ an endpoint of $Y_0$,
the cardinality of $\pi^{-1}(v)$ is $n$. 
 The existence of such families will be shown in a subsequent
subsection. In fact, with suitable choice of parameters, we will show that it is possible 
to choose finite families $\G(n)$ with arbitrarily large cardinality.

Choose a sequence, $\{n(k)\}$, with $n(k)\leq N$ for all $k$.
Using elements of the family $\G^{n(k)}$ as building blocks, we can construct 
inverse systems of metric measure graphs, using the 
procedure described below.  

We begin with a 
 connected metric measure graph $(X_0,d_0,\mu_0)$, with $d_0$ the
length metric,
for which the degree is bounded and such that the
restriction of $(d_0,\mu_0)$ to every edge of $X_0$ is
a copy of $[0,1]$ with the usual Lebesgue measure $\L$.

  Then we 
iterate the following procedure to  construct $X_{k+1}$ and a
map $\pi_k:X_{k+1}\ra X_k$, for every $k$:

\begin{itemize}
\item We choose $n=n(k)\leq N$ and corresponding family $\G(n(k))$ as above.
\vskip1mm

\item We  construct the inverse image $\pi_k^{-1}(V_k)$
of the vertex set $V_k\subset X_k$.  This
is defined to be $V_k\times \{1,\ldots,n\}$, and \texttt{¥}he projection map is the 
projection on the first factor, $\pi_k:V_k\times \{1,\ldots,n\}\ra V_k\subset X_k$.
\vskip1mm

\item For each edge $e_k\subset X_k$, 
we choose 
 a copy of some admissible edge inverse $(Y_0,Y_1,\pi)\in \G(n(k))$, 
with the metrics rescaled 
by $m^{-k}$, the measures rescaled by $\mu_k(e_k)$. Then we identify $Y_0$ with $e_k$
and identify the inverse images of the endpoints
$\{0,1\}=Y_0$ with the inverse images of
the end points of $e_k$ using the identifications of these sets with $\{1,\ldots,n\}$.
Finally, modulo the above identifications, we define the projection map 
$\pi_k:\pi_k^{-1}(e_k)\ra e_k\subset X_k$
to be the projection map $\pi:Y_1\to Y_0$,
\vskip1mm

\item We define $d_{k+1}$ to be the path metric on $X_{k+1}$ which agrees with the
given metric on edges.
\end{itemize}

\begin{lemma}
\label{lem_constructs_admissible_system}
Any inverse system constructed as above is admissible, where the parameters
$\De,\th,C$ depend only on $\{\G(n)\}$ $(n\leq N)$
and the degree bound for $X_0$.
\end{lemma}
\begin{proof}
 Note that $X_0$ is assumed to have bounded degree and
$n(k)\leq N$ for all $k$. Also, for fixed $k$, 
$\{\G(n))\}$ is a finite collection, and each $Y_1\in \G$ is a finite
graph, so that in particular, there is a uniform bound on the degree
for at vertices of elements of $\G(n)$ for all $n$. 
It then follows from (\ref{eqn_degree_1_equal_measure})
that there is a uniform bound on the degree of vertices
of $X_k$ which is independent of $k$. It now clear that Axioms (1) and (2) hold.

Axiom (3) the bound on fibre diameters follows directly
 from the connectedness assumption
(\ref{eqn_connected}).  

Axiom (4), local bounded metric measure geometry, follows
from the finiteness discussion above, together with  (\ref{eqn_degree_1_equal_measure}).
Namely, by (\ref{eqn_degree_1_equal_measure}), for $v_k\in V_k$ and $w_{k+1}\in\pi_k^{-1}(v_k)$
up to scaling of the metric and the measure, the local geometry at $w_{k+1}$ is
the same as the local geometry at $v_k$.

Axiom (5) is immediate from (C), while Axiom (6) follows from (D) and 
(\ref{eqn_degree_1_equal_measure}).
\end{proof}

\subsection{Relaxing some of the conditions}
Next point out some generalizations of the construction above, in which some of the 
conditions are relaxed.

We can relax (\ref{eqn_degree_1_equal_measure}), 
requiring instead that  $\G$ contains  nonempty subsets
of edge inverses satisfying
 (\ref{eqn_degree_1_equal_measure}), and that the rest have the weaker property that
for each vertex $v\in Y_1$ projecting to one  of the  endpoints $0$, $1$,
of $Y_0$, the $\nu_1$ measure
of the edges leaving $v$ is exactly $\frac{1}{mn}$. For subsequent purposes  note
that
in terms of the continuous fuzzy section defined as (\ref{fs3}),
this can be written equivalently 
as follows.   Let $0,1$ denote the vertices of $Y_0=[0,1]$, $\ell\in \{0,1\}$,
 and let $w\in  \pi^{-1}(\ell)$.  Then
$\ell\in\{0,1\}$,
\begin{equation}
\label{equivD}
\D(\ell)(w)=\frac{1}{n}\, .
\end{equation}

The remainder of the discussion of this subsection applies equally well to the general
case (discussed subsequent subsections)
 in which (\ref{equivD}) is replaced by the assumption that for either 
endpoint $\ell\in\{0,1\}$, of $Y_0=[0,1]$, $\D(\ell)(\,\cdot\,)$ is an arbitrary probability measure
taking positive values on every point of $\pi^{-1}(\ell)$; compare (\ref{ulb}).

We may drop the requirement (\ref{eqn_connected}), and instead ask that $\G$ contain
 a nonempty
subset $\G_c$ for which the corresponding $Y_1$ is connected.
 Then to ensure
the point inverses $\pi_k^{-1}(v)$ have controlled diameter, it suffices to 
ensure that the set of edges $e\subset X_k$ for which the inverse image $\pi_k^{-1}(e)$
is chosen  from $\G_c$ forms  a $\tilde Cm^{-k}$ net in $X_k$, where $\tilde C$
 is independent of $k$.

Let $\ell\in \{0,1\}$ denote the endpoints of $Y_0=[0,1]$.
 Denote by $\G_\ell$, $\G_1$, the subset of $\G$ for which every
vertex of $\pi^{-1}(\ell)$ has degree $1$.
Put $\G_0\cap \G_1=\G_{0,1}$.  To ensure the existence of the valence bound
$\Delta$ as
in Axiom (1), we can fix a number $K$, and whenever 
an edge $e\subset X_k$ has a vertex whose degree exceeds $K$ and choose
 the edge inverse from $\G_\ell$, the vertex has degree exceeding $K$
(or from $\G_{0,1}$ if 
both vertices have degree exceeding $K$).

≈
Thus, if $\G$ contains a nonempty subsets $\G_c$, $\G_c\cap G_0$ $\G_c\cap \G_1$ 
$\G_c\cap \G_{0,1}$ we can start by making choices from these subsets  
at sufficiently many edges to form a $\tilde Cm^{-k}$ net, 
 and then, for the remaining edges  make arbitrary choices from $\G$.


\subsection{Admissible edge inverses; the general case}
Next,  we give  the definition of admissible edge inverses in the general case. 

We will retain (A)--(D).  However, we are going to use the reformulation of (C) in 
terms of continuous fuzzy sections.

As discussed in the
special case which we have already treated, the connectedness assumption 
(\ref{eqn_connected}) is 
dropped. (As before, in the inductive construction, for each $k$,
we will assume
as before that the edges with connected $Y_1$ form a $\tilde Cm^{-k}$-net where $\tilde C$
is independent of $k$.)

 For some $N_1$, the inverse
images of the endpoints $\ell\in\{0,1\}$ of $Y_0=[0,1]$ 
are assumed to have cardinalities, $n_0,n_1\leq N_1$, where possibly 
$n_0\ne n_1$. We choose identifications of $\pi^{-1}(\ell)$ with
$1,\ldots,n_\ell$. Let the continuous fuzzy section
$\D$ be defined in terms of $\nu_0,\,\nu_1$
as in (\ref{fs1})--(\ref{fs3}); see also Proposition \ref{equivalent}..
In place of (\ref{equivD}), we simply assume that $\D(\ell)$ is an arbitrary
probability measure on $\pi^{-1}(\ell)$ such that 
\begin{equation}
\label{ulb}
\D(\ell)(w)>c_0'>0\,,
\end{equation}
for all $w\in \pi^{-1}(\ell)$.

 Suppose  we choose to regard
$\D(0)(\, \cdot\,)$ and $\D(1)(\,\cdot\,)$ 
as having been specified. Then as (\ref{fs2}), (\ref{fs3}),
the measure $\nu_1$ provides an extension of $\D$ as a continuous fuzzy section 
to all of $Y_1$.  Conversely, any such extension provides a measure
$\nu_1$ satisfying (C) i.e. the pushforward of $\nu_1$ under $\pi$
is $\nu_0$; see (\ref{fs1}) and Proposition \ref{equivalent} .  
With this much understood, it will be convenient to formulate the rest of 
the discussion of this section in terms of $\D$ (rather than $\nu_1$).

We let
$\G_c\cap \G_0$, $\G_c\cap\G_1$ and $\G_c\cap\G_{0,1}$ 
retain their previous meanings.
Similarly, (\ref{eqn_degree_1_equal_measure}) is dropped
with the proviso that as before, we will only consider collections $\G$ such that 
$\G_c\cap \G_0$, $\G_c\cap\G_1$ and $\G_c\cap\G_{0,1}$ are nonempty,
so that in the inductive construction, we are at liberty make choices from these
subsets when the degree of vertices exceeds a  preselected $K$ and/or to ensure
that edges with connected edge inverses form $\tilde C m^{-k}$-dense subset
of $X_k$.  The existence of such $\G$ is guaranteed by the following Proposition \ref{cee}.


\begin{proposition}
\label{cee}
Assume that the cardinalities  $n_0,n_1$
of $\pi^{-1}(\ell)$ satisfy  $n_\ell\leq N_1$, $\ell\in\{0,1\}$.
  Let $\D$ be specified  arbitrarily on $\pi^{-1}(0)\cup \pi^{-1}(1)$ subject to
the condition that 
(\ref{ulb}) holds for some $c_0'>0$.  Let $\G$ denote the collection of edge inverses
for which $\D$ has the specified restriction to  $\pi^{-1}(0)\cup \pi^{-1}(1)$
and such that in addition, $Y_1$ has $\leq m\cdot N_1$  edges and for all
$i/m \in Y_0'$ and $w\in \pi^{-1}(i/m)$,
\begin{equation}
\label{ecee}
\D(i/m)(w)\geq c_0'\, .
\end{equation}
Then  $\G_c\cap\G_{0,1}$ has cardinality $\geq m-1$.
\end{proposition}
\begin{proof}
Fix some $\underline{i}/m$ be a vertex of $Y_0'$ which is not an end point. 
(Each such choice will determine a different $Y_1$ as in the proposition.)
The combinatorial 
structure of $Y_1$ is specified by stipulating that:

1) $\pi^{-1}(\underline{i}/m)$ consists of
a single vertex $\underline{w}$.

2) For every $w_{0,s}\in \pi^{-1}(0)$ the segment $[0,\underline{i}/m]\subset Y_0'$ from
$v_0$ to $y'$ has a unique lift $\gamma_s$
 with initial point $w_{0,s}$ (and final point 
 $w$).

3) For every $w_{1,t}\in \pi^{-1}(v_1)$ , the segment $[\underline{i}/m,1]\subset Y_0'$ 
has a unique lift $\gamma_t$ with final point $w_{1,t }$ (and initial point  $w$).

$\D$ is given as follows. $\D(\underline{i}/m)(\underline{w})=1$. If
 $w\in \gamma_s$, $w\ne \underline{w}$ then $\D(\pi(w))(w)=\D(0)(w_{0,s})$.  If
$w\in \gamma_t$, $w\ne \underline{w}$ then $\D(\pi(w))(w)=\D(1)(w_{1,t})$. 
\end{proof}

\begin{remark}
Although Proposition \ref{cee} shows the existence of 
$\G$ with $\G_c\cap  \G_{0,1}\ne\emptyset$, it
has the drawback that the combinatorial and metric structure of 
$Y_1$ depends only on $n_0,n_1$. However, as we will see below,
in the general case, we actually do obtain many more examples of admissible
inverse systems that in the simplest special case. 
\end{remark}

\begin{remark}
\label{addingedges}
Fix $\ell\in \{0,1\}$, say $\ell=0$.
There is an obvious $1$-$1$  correspondence between arbitrary admissible
edge inverses 
$(Y_1,d_1,\nu_1)\stackrel{\pi}{\lra}(Y_0,d_0,\nu_0)$
with subdivision parameter $m$ and admissible edge inverses
$(\hat Y_1,\hat d_1,\hat \nu_1)\stackrel{\pi}{\lra}(\hat Y_0,\hat d_0,\hat \nu_0)$
with subdivision parameter $m+1$, such that all vertices in $\pi^{-1}(0)$
have degree $1$.
 Here, after suitable rescaling of the metric
and the measure, we regard $(Y_0,d_0,\nu_0)$ as $\hat \pi^{-1}([1/(m+1),\ldots,1])$.
Also, each vertex in $\hat\pi^{-1}(0)$ is connected to the corresponding vertex in 
$ \hat\pi^{-1}(1/(m+1))$
by a unique edge which projects under $\hat\pi$ to $[0,1/(m+1)]$.
Note that
with the obvious identifications, $\D(\ell)\, |\, \pi^{-1}(\ell)$
remains unchanged, for $\ell$ both $\ell=0$ and $\ell=1$.
 If the edge inverse with subdivision parameter
$m$ is connected, then
so is the new one with subdivision parameter $m+1$.
Of course, the construction can also be done with the end point $\ell=1$,
of with both end points (in which case one obtains an edge
inverse with subdivision parameter $m+2$, for which the inverse images
of both endpoints have degree $1$).
\end{remark}

\subsection{General inductive construction}
Choose constants, $c_0'>0$, 
$0<c_0<{}<c_0'$, $N_1$, $N_2\geq m\cdot N_1$, $\tilde C$ and $K$.
It will be clear that the constants in Axioms (1)--(6),
and hence, the constants in  the doubling condition and Poincar\'e inquality,
can be estimated in
terms of these parameters.

For each vertex $v_k$ of $X_k$, we  specify arbitrarily the cardinality 
$n(v_k)$ of
$\pi_k^{-1}(v_k)$ subject only to  $n(v_k) \leq N_1$.
 We also choose an ordering of $\pi_k^{-1}(v_k)$. 
Finally, we choose an ordering of the vertices of $X_k$.

For each $v_k$ we choose
a probability measure $\D_k$ on $\pi^{-1}_k(v_k)$ such that
\begin{equation}
\label{lbv5}
\D_k(v_k)(v_{k+1})\geq c_0'\, ,
\end{equation}
for all $v_k$, $v_{k+1}\in \pi^{-1}_k(v_k)$.

For each edge $e_k$, the ordering of its vertices induces an identification 
of $e_k$ with $Y_0=[0,1]$ and the specified $\D_k$ on the boundary of $e_k$
induces a probability measure $\D$ on $\pi^{-1}(0)\cup \pi^{-1}(1)$.

 Denote by
$\G$
 the collection of admissible edge inverses with at most
$N_2$ edges,
such that $\D$
on $Y_1$, extends $\D$ on 
 $\pi^{-1}(0)\cup \pi^{-1}(1)$ and such that in addition
\begin{equation}
\label{lbv6}
\D(y)(w)\geq c_0\, ,
\end{equation}
for all $y\in Y_0=[0,1]$ and $w\in \pi^{-1}(y)$.
By Proposition \ref{cee}, $\G_c\cap\G_{0,1}$ has cardinality 
$\geq m-1$; compare however Remark \ref{cardinality}.


Now we proceed mutadis mutandis as we did earlier.
Namely,  for 
$X_k$ select for each edge we select 
an admissible edge inverse
from the corresponding $\G$, subject to the stipulation
that where necessary,  we select from $\G_c$, $\G_c\cap G_0$, 
etc.  In this way the construction of $(X_{k+1},d_{k+1},\mu_{k+1})$ is completed.

\begin{remark}
\label{cardinality}
It will be clear from the discussion of subsequent subsections
that the cardinality of  $\G$ with $\G_c\cap\G_{0,1}$ will tend to infinity
as any of $N_1$, $N_2$ $1/c_0'$, $\tilde C$ or $\K$ goes to infinity.
\end{remark}

\begin{remark}
\label{lowerbound}
It will be seen below that if
we assume that the values of $\D$ on $\pi^{-1}(0)\cup \pi^{-1}(1)$ can all be expressed
as fractions (possibly not in lowest terms) with denominator $d$,
then  $c_0$ can be estimated from below in terms of $c_0', N_2, d$; see Proposition
\ref{entries}.
\end{remark}

\begin{example}
\label{nontriv}
It is easy to construct examples of $\pi_k:X_{k+1}\to X_k$,
such that for {\it no} choice of $\D_k$ on the inverse
images of the vertices, is there an extension of $ \D_k$ to  a continuous
fuzzy section  to $X_{k+1}$. 
For instance,  let $m\geq 2$ and let $X_k$ consist of $2$ oriented edges $e,f$
with a common intitial point $x$ and a common final points $y$.
Let $\pi_k^{-1}(x)=\{p,q\}$ and $\pi_k^{-1}(y)=\{r,s\}$.
Let $\pi^{-1}(e)$ consist of two paths with disjoint interiors, one of
which joins $p$ to $r$ and one of which joins $q$ to $s$.
Let $\pi^{-1}(f)$ consist of a path joining
$p$ to $r$, a path joining $q$ to $r$ and and a path joining $q$ to $s$,
such  that all 3  of these paths have disjoint interiors.

Suppose there exists a continuous fuzzy section $\D_k$.
Using  Axiom (6) (the continuity condition)
and the structure of $\pi_k^{-1}(e)$ it follows that $\D(x)(p)=\D(y)(r)$,
while from the structure of $\pi_k^{-1}(f)$, it follows that 
$D_k(p)>\D_k(r)$.
\end{example}

Having described the inductive construction in  the general case, we
devote the remainder of this section to the construction of large families
of admissible edge inverses.

\subsection{Quotients of edge inverses}
\label{subsec_quotients}

Let $(Y_0, \hat Y_1,\hat \pi)$ be an admissible edge inverse as in the previous subsection
and assume $Y_0'\ne Y_1$.
Form a quotient space $ Y_1$ of $\hat Y_1$, by choosing some edge $e_j'$
in the interior of $Y_0'$
and identifying a pair of distinct inverse images of ${\hat \pi}^{-1}(e_j')$ by the unique isometry
such that the map $\hat \pi$ factors through the quotient map 
$\sigma:\hat Y_1\to  Y_1$ i.e.
$\hat \pi= \pi\circ \sigma$ for some $\pi$. Then if we equip $ Y_1$ with the induced
metric on edges and push-forward measure, $\sigma_*(\hat \nu_1)=\nu_1$, we obtain
a new admissible edge inverse $(Y_0, Y_1,\pi)$.  

Note that 
with the obvious identification of inverse images of end points of $[0,1]$,
we have
\begin{equation}
\label{quo}
\D(\ell)\, |\,  \pi^{-1}(\ell)=\D(\ell)\, |\, \hat \pi^{-1}(\ell)\, .
\end{equation}

We also can also identity a pair of edges in $\hat \pi^{-1}([0,1/m])$ provided
they have the same left-hand end point or a pair in $\hat\pi^{-1}([(m-1)/m,1])$ if
they have the right-hand end point, and do same the construction.

We refer to any edge inverse which
is obtained by starting with   $(Y_0,\hat Y_1,\hat\pi)$ 
and iterating the above constructions
a {\bf quotient of $(Y_0,\hat Y_1,\hat \pi)$}.

Similarly, the above argument can be repeated by identifying vertices in the 
inverse images of interior vertices of $Y_0'$
 in place of edges.
We also refer to the result as a {\bf quotient of $(Y_0,Y_1,\pi)$}.

In particular, the quotient construction can be applied to a
an admissible edge inverse as in Proposition \ref{cee}.  More importantly,
it can be applied to
``special admissible edge inverse'' as defined in the next section.
In fact, we will show that every admissible edge inverse arises
as a quotient of a special one.

\begin{remark}
\label{r:axioms}
It is easy to verify that both 
$(Y_1,d_1,\nu_1) \stackrel{\sigma}{\lra}
(Y_0,d_0,\nu_0)$
and
$ (\hat Y_1,\hat d_1,\hat \nu_1)\stackrel{\sigma}{\lra}
(Y_1,d_1,\nu_1)$, satisfy Axioms (1)--(6).
\end{remark}

\subsection{Special admissible edge inverses}

In this section we define a class of admissible edge inverses (called ``special'')
whose combinatorial and metric 
classification can be reduced to the problem of describing the supports
of all probability matrices with specified marginals. For the case in which
the marginals take rational values, this can be done in terms of the Birkoff-Von Neumann
theorem.  For each possible support, the Birkoff-Von Neumann
theorem also provides a canonical representative probability matrix
whose entries have a definite lower bound. This is required to control
the measure of the associated special edge inverse.

It will be clear that the cardinality of the collection of combinatorially
distinct admissible edge inverses with specified
marginals will be arbitrarily large if the parameters on
which the associated matrix depends  are sufficienly large. Moreover,
by taking quotients as in the last section   one obtains a much larger
class of combinatorially distinct examples.  
In a subsequent subsection we will see that all examples
of admissible edge inverses arise as quotients of special ones.

A {\bf special edge inverse}  is an edge inverse such that:

\no
1.   Each component of $\pi^{-1}((0,1))$ 
is an open interval $\gamma$.  (Thus, the closures of to
such components intersect only at some point of $\pi^{-1}(0)$
and some point of $\pi^{-1}(1)$.)
\vskip2mm

2. If $\gamma$ is a component of $\pi^{-1}((0,1))$ then
 $\D(\pi(w))(w)$ is the same for all $w\in \gamma$.
\vskip2mm


 For $w\in \gamma$ as above, we call  $\D(\pi(w))(w)$ the
{\bf weight} of $\gamma$.

Suppose we are give a special admissible edge inverse.
Let $n_1,n_2$ denote the cardinalities of $\pi^{-1}(0)=\{w_{0,t}\}$ and
$\pi^{-1}(1)=\{w_{1,s}\}$ 
respectively.
Define an $n_1\times n_2$ probability matrix 
$P_{s,t}$, whose $s,t$-th entry is the sum of the weights of all 
those $\gamma$ as above with initial point $w_{0,t}$ and final point $w_{1,s}$.
 Then $P_{s,t}$  has the property that its marginals are given by 
$\D(0)(w_{0,t})$ and $\D(1)(w_{1,s})$.

Conversely, suppose we are given an $n_1\times n_2$ probability matrix $P_{s,t}$
 and positive integers $c_{s,t}$ for each nonzero entry $p_{s,t}>0$.  
Then there is a unique 
special admissible edge inverse with $c_{s,t}$ paths $\gamma$ connecting
$w_{0,t}$ to $w_{1,s}$ for each $(s,t)$, such that each such 
$\gamma$ connecting $w_{0,s}$
and $w_{1,t}$ has weight $p_{s,t}/c_{s,t}$.
The resulting special edge inverse has the property that 
$\D(0)(w_{0,t})$ and $\D(1)(w_{1,s})$
 are given by the marginals of $P_{s,t}$.

Therefore, we get the following.

\begin{proposition}
The combinatorial classification of special admissible edge inverses with a specified
$\D$ on the inverse images of the end points, is equivalent to the classification of
the supports of probability matrices with specified marginals.
\end{proposition}

Consider the simplest special case treated at the beginning of this section,
 in which $n_1=n_2=n$ and marginals, all equal to
$\frac{1}{n}$. In that case, $P_{s,t}$ is a so called doubly stochastic matrix and
there is a representation theorem, the Birkoff-Von Neumann 
theorem,   which describes all such matrices.  
\begin{theorem} (Birkoff-Von Neumann)
\label{bvn}
The space of all doubly  stochastic matrices has dimenson
$(n-1)\times (n-1)$. Any such matrix is a convex combination of permutation matrices.
\end{theorem}

\begin{remark}
\label{entries}
Note that while the combinatorial a metric structure of the associated 
special admissible edge inverse is determined by the support of the 
corresponding probability matrix $P_{s,t}$, a bound on $\D$
(or equivalently on the ratio of $\nu_1$ to Lebesgue measure) is determined
by a lower bound on the actual entries and the constants $c_{s,t}$,
(which are bounded in terms of $N_2)$.  

For the case of doubly stochastic matrices
the support is determined just by the collection of nozero coefficients 
representation in the representation supplied by the Birkoff-Von Neumann theorem.  
By choosing all such coefficients to be equal, we obtain
matrix with the given support and a definite lower bound on the entries. 
Note that in the application
to edge inverses, it is the entries which determine $\D_{k+1}$.  Therefore, in what follows,
we will always assume without further mention that this canonical choice has been made.

Below we will show that the classification of probability matrices with 
rational entries can also be reduced to
the case of doubly stochastic matrices described above.  Therefore,
we have canonical representatives with a lower bound on the entries
for each possible support in this case as well.
\end{remark}

Given a $d\times d$ doubly stochastic matrix, for some integer $a$
replace the first $a$ rows by a single row which is equal to their sum
and whose column marginal remains unchanged.  By suitably iterating this 
operation we obtain a matrix whose row marginals are any sequence of
length $<d$,
of positive rational numbers with denominator $d$ whose sum is equal to $1$.
Then we can repeat the same operations with columns in place of rows.
In this way we can obtain a matrix with any specified row and column
marginals all of whose entries are rational numbers with denominator $d$.
(We do not assume that these fractions are in lowest terms.)

In fact, every probability matrix with rational marginals such that
every entry has denominator $d$ arises in this way. To see this,
let $P=(p_{s,t})$ denote an $n_1\times n_2$ probability 
matrix with rational entries and marginals
$(\rho_s)$ and $(\tau_t)$. Let $d$ denote 
the least common denominator for $\{\rho_s\}\cup\{\tau_t\}$. 
Write $\rho_s=\alpha_s/d$, $\tau_t=\beta_t/d$.  For  each $s$,
replace the $s$-th row by $\alpha_s$ identical rows, each with entries $p_{s,t}/\alpha_s$.
This operation yields a $d\times n_2$ probability matrix whose for which the row marginal has entries
$1/d$ and whose column marginal remains unchanged.
 Now by repeating this operation with columns in place of 
rows, we obtain a doubly stochastic $d\times d$ probability matrix $\tilde P$ i.e.  
all entries of the row and column marginals are  equal to $1/d$.
Clearly, the original matrix $P_{s,t}$ can be obtained from  the doubly stochastic matrix
$\tilde P$ as in the previous paragraph.

In this sense, we have reduced the representation of arbitrary probability matrices
with rational marginals to the 
Birkoff-Von Neumann theorem.

\begin{remark}
Suppose we are given the support of an $n_1\times n_2$ probability matrix
and a specified row marginal $(\rho_s)$.Then there is
a unique probability matrix $P$ denote  with the given row marginal such that all 
entries in any
given row are the same.  

As a consequence, given $X_k$ and a maximal collection of disjoint
edges $\C=\{e_k\}$, the metric measure structure of the special edge inverses over
these $e_k$ and in particular, the combinatorial structure,
can be specified arbitrarily, the only caveat being that when 
necessary, we choose an arbitrary element of $\G_0,\G_1$ or $\G_{0,1}$; see
Remark \ref{addingedges} and compare Remark \ref{nontriv}. 
The corresponding collection of row and column marginals determins
$\D_k$ on $\pi_k^{-1}(v_k)$,  all vertices $v_k$ of $X_k$.  Then the edge inverses
of the remaining edges can be chosen as in the general inductive step. (The 
required $\tilde Cm^{-k}$-dense set of connected edge inverses can be
chosen from either $\C$ or its complement.)
\end{remark}

\subsection{Arbitrary edge inverses are quotients of special ones}
\label{pm}

We now show:
\begin{proposition}
\label{qsei}
 For any admissible edge inverse 
$(Y_1,d_1,\nu_1)\stackrel{\pi}{\lra}(Y_0,d_0,\nu_0)$,
there is a (canonically associated) special admissible edge inverse,
$( \hat Y_1,\hat d_1,\hat \nu_1)\stackrel{\hat \pi}{\lra}(Y_0, d_0, \nu_0)$,
of which $(Y_1,d_1,\nu_1)\stackrel{\pi}{\lra}(Y_0,d_0,\nu_0)$ is the quotient.
\end{proposition}
\begin{proof}
Regard, $Y_0'$ as a path $\gamma'_0$, and let $\Gamma$ denote the collection of
lifts to $Y_1$, as in Section \ref{mlp}.
 For each $\gamma_1\in \Gamma$ take
a copy $I_{\gamma_1}$ of $Y_0'$ and
form the quotient space $\hat Y_1$ of
$\bigcup_{\gamma_1\in\Gamma}I_{\gamma_1}$  by the equivalence relations
generated as follows: For
all $\gamma_{1,1}, \gamma_{1,2}\in \Gamma$, identify 
$I_{\gamma_{1,1}}(0)$ with $I_{\gamma_{1,2}}(0)$ if and only if
$\gamma_{1,1}(0)=\gamma_{1,2}(0)$.  Similarly, identify
$I_{\gamma_{1,1}}(1)$ with $I_{\gamma_{1,2}}(1)$ if and only if
$\gamma_{1,1}(1)=\gamma_{1,2}(1)$.  Give $\hat Y_1$ the path metric
on components.  There is a natural projection $\sigma:\hat Y_1\to Y_1$.
Put  $\hat\pi=\sigma\circ \pi$. Then the restriction of 
$\sigma$ to $\hat\pi^{-1}(0)\cup  \hat\pi^{-1}(1)$ is $1$-$1$ and
onto $\pi^{-1}(0)\cup  \pi^{-1}(1)$. 

It should be clear that the only remaining point
is to specify the measure $\hat\nu_1$ such that $\sigma_*(\hat\nu_1)=\nu_1$.  
To this end, we use an appropriate
continuous fuzzy section $\hat \D_0$ of $\hat \pi$ defined as follows.
For all $y_0'$ in the interior of 
$Y_0'$, $\gamma_1\in\Gamma$  and $y_1\in \hat\pi^{-1}\cap I_{\gamma_1}$,
we put
\begin{equation}
\label{hatDdef}
\hat\D(y_0')(y_1)=\Omega(\gamma_1)\, ,
\end{equation}
where $\Omega$ is the probability measure on $\Gamma$ defined in (\ref{measdef}).
in 
Then there is a unique extension of $\hat \D_0$ to a continuous 
fuzzy section of $\hat\pi$ on all of $Y_0'$.
It then follows from (\ref{edge}) that 
$\sigma_*(\hat \D_0)=\D_0$, which implies
 $\sigma_*(\hat \nu_1)=\nu_1$.
This suffices to complete the proof.
\end{proof}

\section{Analytic dimension 1}
\label{s:ad}

In this section, we assume familiarity with certain material from 
\cite{cheeger} (see  in particular Sections 2 and 4) including
 the fact that a PI space $(X,d,\mu)$  has a measurable cotangent
bundle $TX^*$. In particular, there is a $\mu$-a.e. well
defined fibre $TX^*_x$. We also, use the Sobolev spaces $H_{1,p}$ and the fact that they
are reflexive.

We show:
\begin{theorem}
\label{ad1}
If  $(X_\infty,d_\infty,\mu_\infty)$ is the measure Gromov-Hausdorff limit of
an admissible inverse system, then
 the dimension of the fibre of the cotangent bundle is 1 $\mu$-a.e..
\end{theorem}
\begin{proof}
Without essential loss of generality, we can assume $X_0=\R$. (Otherwise, we restrict attention
to the inverse image of each individual open edge in $X_0$.)

Let $f:\R\to \R$ denote the identity map viewed as a $1$-Lipschitz function
on $\R$. Let $f_i=f\circ \pi_{i-1}:X_i\to \R$.  From Axioms (1) and (2) in the definition 
of admissible inverse systems, it is clear that $df_i$ defines a trivialization of
the cotangent bundle of $X_i$.
Let $\pi_i^\infty:X_\infty\to X_i$ denote the natural projection and set $f_\infty=f\circ\pi_i^\infty$.

It is easy to see that any $L$-Lipschitz function $h:X_\infty \to \R$, 
is the uniform limit as $i\to\infty$ of $2L$-Lipschitz functions of the form 
$h_i=\tilde h_i\circ \pi_i^\infty$ where $\tilde h_i$ is a $2L$-Lipschitz
function on $X_i$.  It follows that $d\tilde h_i$ is a bounded measurable function times $df_i$ 
and hence, that  $dh_i$ is a bounded measurable function times $df_\infty$.
Clearly, the same holds for any finite linear combination of the $h_i$.

By the reflexivity of the Sobolev space $H_{1,p}$ 
it follows that there is a a sequence $\hat h_i$ of such combinations which converges
to $h$ in $H_{1,p}$. It follows that $df$ is a bounded measurable function times
$df_\infty$, which suffices to complete the proof.
\end{proof}

\section{Bi-Lipschitz nonembedding in Banach spaces with the RNP}
\label{s:blnernp}
Recall that a Banach space $\mathcal V$ is said to have the Radon-Nikodym Property
if every Lipschitz map $f:\R\to\mathcal V$ is differentiable almost everywhere. Separable dual spaces
such as $L_p$ for $1<p<\infty$ and $\ell_1$ have the Radon-Nikodym Property but $L_1$ 
does not.

In this section we show that except in degenerate cases, the Gromov-Hausdorff limit $(X_\infty,d_\infty)$ of
an admissible inverse system does not bilipschitz embed in any Banach spaces with the
Radon-Nikodym property. However it follows directly from the main result of \cite{ckellone} these spaces 
do bilipschitz embed in $L_1$. 

Since by Theorem \ref{t:pi}, $(X_\infty,d_\infty,\mu_\infty)$ is a PI space, according
to \cite{ckdppi},  it will suffice to give conditions
on  $(X_\infty,d_\infty,\mu_\infty)$ which guarantee that for a subset of positive 
$\mu_\infty$ measure, some tangent cone is not bilipschitz to a Euclidean space. According to the following
lemma, in our situation, the only possibility for the dimension of this Euclidean space is $1$; compare
Corollary \ref{cor_top_dim_1}.

Let $(X_i,\pi_i,\mu_i)$ denote an admissible inverse system with subdivision parameter $m\geq 2$.
Let $V_i^{\geq 3}\subset X_i$ denote the set of vertices of $X_i$ with degree
at least $3$.  Given a vertex $v_i\in X_i$, we define the {\bf halfstar of $v_i$ in $X_i$}
to be the union $\halfstar(v_i,X_i)\subset X_i$
of the segments of length $\frac12 m^{-i}$ emanating from $v_i$.

\begin{lemma}
\label{lem_tangent_cone_structure}
Let $(X_i,\pi_i,\mu_i)$ denote an admissible inverse system with subdivision parameter $m\geq 2$.

\begin{enumerate}
\item Let $x_\infty\in X_\infty$ and assume $\pi_i^\infty(x_\infty)$ is a vertex of $X_i$.
Then there is a subset $Y_\infty\subset X_\infty$
which projects isometrically under $\pi_i^\infty$ to the halfstar $\halfstar(\pi_i^\infty(x_\infty),X_i)$.
\vskip2mm

\item Let $x_\infty\in X_\infty$. Then every tangent cone   of $X_\infty$
at $x_\infty$ is 
homeomorphic to $\R$ if and only if every such tangent cone 
is isometric to $\R$. This holds if and only if
$$
\liminf_{i\ra\infty}\; m^i\cdot d_i(\pi_i^\infty(p_\infty),V_i^{\geq 3})=\infty\,.
$$
\item For all $x_\infty\in X_\infty$, every tangent cone at $x_\infty$ has topological dimension $1$. 
\end{enumerate}
\end{lemma}
\begin{proof}
(1).  Let $Y_i=\halfstar(\pi_i^\infty(x_\infty),X_i)$.
Given a geodesic path of length $\frac12 m^{-i}$ emanating from $\pi_i^\infty(x_\infty)$, we can
lift it to a path in $X_{i+1}$ starting at $\pi_{i+1}^\infty(x_\infty)$; see the discussion of
Axiom (2) in Section \ref{intro}.
By taking the union of one such lift for each
path, we obtain a lift $Y_{i+1}$ of $Y_i$.  Iterating this produces
a compatible sequence $\{Y_j\subset X_j\}_{j\geq i}$ that projects isometrically
to $\halfstar(\pi_i^\infty(x_\infty),X_i)$ under the projections $\pi_i^j:X_j\ra X_i$. 
Then the inverse limit of $\{Y_j\}$ is the desired subset.

(2).  If $\liminf_{i\ra\infty}\; m^i\cdot d_i(\pi_i^\infty(p_\infty),V_i^{\geq 3})=D<\infty$, 
then using path lifting one gets sequences $i_j\ra \infty$, $\{x_{j,\infty}\}\subset X_\infty$, 
such that $\pi_{i_j}^\infty(x_{j,\infty})\in Y_{i_j}^{\geq 3}$, and $d(x_{j,\infty},p_\infty)<2D m^{-i_j}$.
Then by (1), for every $j$ the rescaled pointed space
$(X_\infty,m^{i_j}d_\infty,p_\infty)$
contains an isometric copy of a ``tripod''  of size $\frac{1}{2}$
within the ball $B(p,2(D+1))\subset \,(X_\infty, m^{i_j}d_\infty)$. (By a tripod
of size $\frac{1}{2}$, we mean $3$ line segments, each of length $\frac{1}{2}$, emanating from 
a single point, equipped with the path metric.)
  Therefore
any pointed Gromov-Hausdorff limit of a subsequence of the sequence
$\{(X_\infty,m^{i_j}d_\infty,
p_\infty)\}_{j}$ will contain an isometric copy of such a tripod, and hence
cannot be homeomorphic to $\R$.

Suppose conversely, that $\liminf_{i\ra\infty}\; m^i\cdot d_i(\pi_i^\infty(p_\infty),V_i^{\geq 3})=\infty$.
 Let
 $D_i=m^i\cdot d_i(\pi_i^\infty(p_\infty),V_i^{\geq 3})$.   Then $D_i\ra \infty$, so we can choose
sequences $\{j_i\}$, $\{R_i\}$ such that: 
\begin{itemize}
\item $j_i-i\ra \infty$ and $R_i\ra \infty$
as $i\ra \infty$.
\vskip1mm

\item
 $B_{m^{-i}R_i}(\pi_{j_i}^\infty(p_\infty))\subset X_{j_i}$ contains only degree $2$ vertices
and is therefore isometric to an interval. 
\end{itemize}
It follows that the pointed sequence
$\{(X_{j_i},m^{i}d_{j_i},\pi_{j_i}^\infty(p_\infty))\}$  converges to $(\R,0)$ in the pointed Gromov-Hausdorff
topology,
 and also to any tangent cone at
$(X_\infty,p_\infty)$, since the projection map 
$\pi_{j_i}^\infty: (X_\infty,m^id_\infty)\to (X_{j_i},m^i d_{j_i})$
 is a $Cm^{i-j_i}$-Hausdorff approximation.

(3).    It is easy to see that up to rescaling of the metric, a tangent cone 
at a point of $X_\infty$
is itself the
pointed Gromov-Hausdorff limit 
of an admissible inverse system.  Then, by
Corollary \ref{cor_top_dim_1}, it follows that every 
such tangent cone has topological dimension 1.
\end{proof}

Thus  we obtain the following:

\begin{theorem}
\label{blne}
If $\{(X_i,d_i,\mu_i)\}$ is an admissible inverse system, and a positive $\mu_\infty$
measure set of points $x_\infty\in X_\infty$ satisfy
\begin{equation}
\label{eblne}
\liminf_{i\ra\infty}\; m^i\cdot d_i(\pi_i^\infty(x_\infty),V_i^{\geq 3})<\infty\,,
\end{equation}
then $(X_\infty,d_\infty)$ does not bilipschitz embed in any Banach space
with the Radon-Nikodym Property.
\end{theorem}
\begin{proof}
By Lemma \ref{lem_tangent_cone_structure}, any tangent cone at such a point $x_\infty$
has topological dimension $1$, and contains an isometric copy of a tripod.  Therefore
it cannot be homeomorphic to $\R^n$ for any $n$.  
Now \cite{ckdppi} implies that $X_\infty$ does not bilipschitz embed in any Banach
space with the Radon-Nikodym Property.
\end{proof}

\begin{remark}
Examples which fail to satisfy (\ref{eblne}) are ``degenerate'' in an obvious sense.
\end{remark}

\section{Higher dimensional inverse systems}
\label{higher}

In this section we consider  higher dimensional inverse systems, where each
$X_i$ is a cube complex.  We would like to point out  
that there are other ways of generalizing
to higher dimension; in particular, one can construct examples of inverse systems
where $X_0$ is the Heisenberg group with the Carnot metric,  the mappings
$\pi_i:X_{i+1}\ra X_i$ are ``branched mappings'', and the inverse limit is a PI space.

We recall that the {\bf star} of a face $C$ in a polyhedral complex $X$ is the union
$\Star(C,X)$ of the faces containing it.  A {\bf gallery} in an $n$-dimensional
polyhedral complex is a sequence $C_0,\ldots,C_N$
of top dimensional faces where $C_{i-1}\cap C_i$ is a codimension $1$ face for all $1\leq i\leq N$.

Fix $n\geq 1$.  We consider an inverse system

\begin{equation}
\label{is_n}
X_0\stackrel{\pi_0}{\longleftarrow}\cdots \stackrel{\pi_{i-1}}{\longleftarrow}X_{i}
\stackrel{\pi_{i}}{\longleftarrow}
\cdots\, .
\end{equation}
such that each $X_i$ is a connected cube complex equipped with a  path metric $d_i$ and a measure
$\mu_i$, such that the following conditions hold, for some constants
$2\leq m\in \Z$, $\De$, $\th$, $C\in (0,\infty)$ and
every $i\in \Z$\,:

\begin{enumerate}
\setlength{\itemsep}{.5ex}
\item   (Bounded local metric geometry) $(X_i,d_i)$ is a nonempty connected cube complex that is a union of $n$-dimensional faces isometric to the
$n$-cube $[0,m^{-i}]^n$  (with respect to the path metric $d_i$), such that every link contains at most $\De$ faces.

\item (Simplicial projections are open)  If $X_i'$ denotes the  cube complex obtained by
subdividing each cube of $X_i$ into
$m^n$ subcubes isometric to $[0,m^{-(i+1)}]^n$, then $\pi_i$ induces a map
$\pi_i:(X_{i+1},d_{i+1})\ra (X_i',d_i)$ which is open, cellular (with respect to the cube structure), and an isometry on every face.
\item (Gallery diameter of  fibers is controlled)
For every  $x_i\in X_i'$, any two points in the inverse image $\pi_i^{-1}(x_i)\subset X_{i+1}$ can be joined by a gallery of $n$-cubes $C_0,\ldots,C_N$,
where $N\leq \De$. \item (Bounded local metric measure geometry.)
The measure $\mu_i$ restricts to a constant multiple of Lebesgue measure on each $n$-cube $C_i\subset X_i$, and  $\frac{\mu_i(C_{i,1})}{\mu_i(C_{i,2})}\in [C^{-1},C]$ for any two adjacent $n$-cubes $C_{i,1},C_{i,2}\subset X_i$.
\item (Compatibility with projections) $$
(\pi_i)_*(\mu_{i+1})=\mu_i\, ,
$$
where $(\pi_i)_*(\mu_{i+1})$ denotes the pushforward of $\mu_{i+1}$ under $\pi_i$.
\item (Continuity across codimension $1$ faces) For every pair of
codimension $1$ faces $c_i'\subset X_i'$, and $c_{i+1}\subset \pi_i^{-1}(v_i')$,
the quantity
\begin{equation}
\label{con12}
\frac{\mu_{i+1}(\pi_i^{-1}(C_i')\cap \Star(c_{i+1},X_{i+1}))}{\mu_i(C_i')}
\end{equation}
is the same for all $n$-cubes $C_i'\subset \Star(c_i',X_i')$. 
\vskip2mm
\end{enumerate}

The biggest difference between the axioms above and Definition \ref{admissible} is in Axiom (3) 
above, where path diameter has been replaced by gallery diameter.  Note that the gallery diameter
 is the same a
path diameter in the case of graphs.  A bound on the  path diameter would be sufficient
to verify most of the properties that hold for admissible inverse systems
of graphs.  However,
it is not sufficient to recover the main result --- the $(1,1)$-Poincare inequality
as the following example illustrates.

\begin{example}
Consider the $2$-dimensional  inverse system with  subdivision
parameter $m=2$, where:
\begin{itemize}
\item   $X_0$ is the unit square
$[0,1]^2$.
\item  $X_1$ is obtained by taking two copies of the subdivided complex $X_0'$ and gluing 
them together along their central vertices. \item All projection maps $\pi_i:X_{i+1}\ra X_i'$ with $i>0$ 
are isomorphisms.
\end{itemize}
Then $X_\infty$ is isometric to $X_1$, and does not satisfy a $(1,1)$-Poincare inequality;
this is because the gluing locus --- a singleton --- has zero $1$-capacity.
\end{example}

Let $X_\infty$ be the inverse limit of an inverse system satisfying 
(1)-(6) above.
The proof of the Poincar\'e inequality for $X_\infty$
using path families carries over in a straightforward
way, when one uses geodesic paths that intersect each $n$-cube $C$ in a segment parallel to an edge of $C$.
So does the proof using continuous fuzzy sections.

\begin{remark}
What is essential   in Axioms (1) and (4) 
is that they imply that $X_i$ is 
doubling and satisfies a $(1,1)$-Poincar\'e inequality
on scale $m^{-i}$.  In the above example, this doesn't
hold.  However, if Axiom (4) is appropriately modified, then Axiom (3) can be left 
as is.
\end{remark}

\bibliography{piex}
\bibliographystyle{alpha}
\end{document}